\input xypic

\documentclass[10pt]{amsart}
\usepackage{graphicx,amssymb,amsfonts,amsmath,amsthm,newlfont}
\usepackage{epsfig}
\newtheorem{thm}{Theorem}[section]
\newtheorem{cor}[thm]{Corollary}
\newtheorem{lem}[thm]{Lemma}
\newtheorem{prop}[thm]{Proposition}

\theoremstyle{definition}
\newtheorem{defn}[thm]{Definition}
\newtheorem{conj}[thm]{Conjecture}

\newtheorem{ex}[thm]{Examples}

\newtheorem{example}[thm]{Example}
\theoremstyle{remark}
\newtheorem{rem}[thm]{Remark}
\numberwithin{equation}{section}

\font \rus= wncyr10

\newcommand{\Z}{\mathbb Z}
\newcommand{\C}{\mathbb C}
\newcommand{\Pro}{\mathbb P}

\newcommand{\R}{\mathbb R}
\newcommand{\N}{\mathbb N}
\newcommand{\Q}{\mathbb Q}

\newcommand{\Hom}{\hbox{Hom}}

\newcommand{\I}{\mathbb{I}}

\newcommand{\To}{\longrightarrow}
\newcommand{\sha}{\, \hbox{\rus x} \,}

\newcommand{\Mod}{\mathfrak{M}}\newcommand{\Modf}{\mathfrak{M}^\delta}

\newcommand{\Sym}{\mathfrak{S}}

\newcommand{\ord}{\mathrm{ord}}
\newcommand{\A}{\mathbb{A}}
\newcommand{\Res}{\mathrm{Res}}
\newcommand{\PSL}{\mathrm{PSL}}
\newcommand{\gr}{\mathrm{gr}}

\newcommand{\MT}{\mathcal{MT}}
\newcommand{\FMT}{\mathcal{M}}

\newcommand{\F}{\mathfrak{F}}
\newcommand{\cfl}{\langle}
\newcommand{\cfr}{\rangle}

\newcommand{\Lie}{\mathfrak{L}}

\newcommand{\spc}{\, \, ; \,\,}

\begin{document}
\title{The algebra of cell-zeta values}
\author{Francis Brown, Sarah Carr, Leila Schneps}
\maketitle

\begin{abstract}
In this paper, we introduce {\it cell-forms} on $\Mod_{0,n}$,
which are top-dimensional differential forms diverging along the 
boundary of exactly one cell (connected component) of the real moduli space
$\Mod_{0,n}(\R)$. We show that the cell-forms generate the top-dimensional
cohomology group of $\Mod_{0,n}$, so that there is a natural duality
between cells and cell-forms. In the heart of the paper, we determine
an explicit basis for the subspace of differential forms which 
converge along a given cell $X$.   The elements of this basis are
called {\it insertion forms}, their integrals over $X$ are real
numbers, called {\it cell-zeta values}, which generate a $\Q$-algebra
called the cell-zeta algebra.  By a result of F. Brown,
the cell-zeta algebra is equal to the algebra of multizeta values.
The cell-zeta values satisfy a family of simple quadratic relations coming
from the geometry of moduli spaces, which leads to a natural definition
of a formal version of the cell-zeta algebra, conjecturally isomorphic to
the formal multizeta algebra defined by the much-studied double shuffle
relations.
\end{abstract}

\noindent Mathematics Subject Index 2000: 11Y40, 14Q15, 68W30


\section{Introduction}

Let $n_1,\ldots, n_r\in \N$ and suppose that $n_r\geq 2$. The
multiple zeta values (MZV's)
\begin{equation}\label{MZV}
\zeta(n_1,\ldots, n_r) = \sum_{0<k_1<\ldots<k_r} {1 \over k_1^{n_1}
\ldots k_r^{n_r}}\in \R\ ,\end{equation} were first defined by
Euler, and have recently acquired much importance in their relation
to mixed Tate motives. It is conjectured that the periods of all
mixed Tate motives over $\Z$ are expressible in terms of such
numbers. By an observation due to Kontsevich, every multiple zeta value
can be written as an iterated integral:
\begin{equation}\label{itint}
\int_{0\leq t_1\leq \ldots \leq t_\ell \leq 1} {dt_1 \ldots dt_\ell
\over (\varepsilon_{1}-t_1)\ldots(\varepsilon_\ell-t_\ell) }\ ,
\end{equation} where $\varepsilon_i \in \{0,1\}$, and
$\varepsilon_1=1$ and $\varepsilon_\ell=0$ to ensure convergence,
and $\ell=n_1+\cdots +n_r$. The iterated integral $(\ref{itint})$ can
be considered as a period on $\Mod_{0,n}$ (with $n=\ell+3$), or 
a period of the motivic fundamental group of
$\Mod_{0,4}=\Pro^1\backslash\{0,1,\infty\}$, whose de Rham
cohomology $H^1(\Mod_{0,4})$ is spanned by the forms ${dt \over t}$
and ${dt \over 1-t}$ \cite{De1}, \cite{D-G}. One proves that the multiple
zeta values satisfy two sets of quadratic relations \cite{Ch1}, \cite{Ho},
known as the regularised double shuffle relations, and it has been
conjectured that these generate all algebraic relations between
MZV's \cite{Ca2}, \cite{W}.  This is the  traditional point
of view on multiple zeta values.

On the other hand, by  a general construction due to Beilinson, one
can view the iterated integral $(\ref{itint})$ as a period integral
in the ordinary sense, but this time  of the $\ell$-dimensional
affine scheme
$$\Mod_{0,n} \simeq (\Mod_{0,4})^\ell \backslash \{\hbox{diagonals}\} = \{(t_1,\ldots, t_\ell): t_i\neq 0,1\ , t_i\neq t_j\}\ ,$$
where $n=\ell+3$. This is the moduli space of curves of genus $0$
with $n$ ordered marked points. Indeed, the open domain of
integration $X=\{0< t_1< \ldots < t_\ell < 1\}$ is one of the
connected components of the set of real points $\Mod_{0,n}(\R)$, and
the integrand of $(\ref{itint})$ is a regular algebraic form in
$H^\ell(\Mod_{0,n})$ which converges on $X$.  Thus, the study of
multiple zeta values leads naturally to the study of all periods on
$\Mod_{0,n}$, which was initiated by Goncharov and Manin
\cite{Br2}, \cite{G-M}. These periods can be written
\begin{equation}\label{introint}
\int_X \omega\ , \quad \hbox{ where } \omega \in H^\ell(\Mod_{0,n})
\hbox{ has no poles along } \overline{X}\ .\end{equation}
  The
general philosophy of motives and their periods \cite{Ko-Za}
indicates that one should study relations between all such
integrals.  This leads to the following problems:
\begin{enumerate}
  \item Construct a good basis of all logarithmic
  $\ell$-forms $\omega$ in $H^\ell(\Mod_{0,n})$ whose integral over the cell
  $X$ converges.
  \item Find all relations between the integrals $\int_{X} \omega$
  which arise from natural geometric considerations on the moduli spaces
  $\Mod_{0,n}$.
\end{enumerate}
In this paper, we give an explicit solution to $(1)$, and a family
of relations which conjecturally answers $(2)$. Firstly, we give an explicit
description of a basis of the subspace of $H^\ell(\Mod_{0,n})$ of
forms convergent on the standard cell, in terms of the combinatorics of 
polygons.  (Note that the idea of connecting differential forms with 
combinatorial structures has previously been explored from different aspects,
in \cite{G-G-L} and \cite{Te2} for example.)  The corresponding integrals are 
more general than 
(\ref{itint}), although Brown's theorem \cite{Br2} proves that they do occur 
as $\Q$-linear combinations of multiple zeta values of the form (\ref{itint}).

For $(2)$, we explore a new family of quadratic relations, which we call 
{\it product map relations}, because they arise from products of forgetful 
maps between moduli spaces.   To this family we add two other simpler
families; one arising from the dihedral subgroup of automorphisms of 
$\Mod_{0,n}$ which stabilise $X$, and the other from a basic identity
in the combinatorics of polygons.  These families are sufficiently 
intrinsic and general to motivate the following conjecture, which we have
verified computationally up through $n=9$.

\vspace{.3cm}
\noindent {\bf Conjecture.} {\it The three families of relations between 
integrals (given explicitly in definition \ref{threerels}) 
generate the complete set of 
relations between periods of the moduli spaces $\Mod_{0,n}$.}
\vskip .3cm
\noindent {\it Acknowledgements:} The authors wish to extend warm thanks
to the referee, who made a remarkable effort and provided numerous helpful
observations.

\subsection{Main results}\label{thing}
We give a brief presentation of the main objects introduced in this
paper, and the results obtained using them. \vskip .2cm

Recall that Deligne-Mumford constructed a stable compactification $\overline\Mod_{0,n}$ of
$\Mod_{0,n}$, such that $\overline\Mod_{0,n}\setminus\Mod_{0,n}$ is
a smooth normal crossing divisor whose irreducible components
correspond bijectively to partitions of the set of $n$ marked points
into two subsets of cardinal $\ge 2$ \cite{D-M}, \cite{Kn}. The real part
$\Mod_{0,n}(\R)$ of $\Mod_{0,n}$ is not connected, but has $n!/2n$
connected components (open cells) corresponding to the different
cyclic orders of the real points $0,t_1,\ldots,t_\ell,1,\infty\in
{\mathbb P}^1(\R)$, up to dihedral permutation \cite{Dev1}. 
Thus, we can identify cells with $n$-sided polygons with edges
labeled by $\{0,t_1,\ldots, t_\ell, 1, \infty\}$.  In the
compactification $\overline{\Mod}_{0,n}(\R)$, the closed cells have
the structure of associahedra or Stasheff polytopes; the boundary of
a given cell is a union of irreducible divisors corresponding to
partitions given by the chords (cf. definition \ref{defchords})
in the  associated polygon. The
standard cell is the cell corresponding to the standard order we
denote $\delta$, given by $0<t_1<\ldots<t_\ell<1$.   We write
$\Mod_{0,n}^\delta$ for the space 
$$\overline\Mod_{0,n}\setminus \{\mbox{all boundary
divisors of }\Mod_{0,n}\mbox{ except those bounding the standard cell}\}.$$
This is a smooth affine scheme introduced in \cite{Br2}.

\subsubsection{Polygons.} Since a cell of $\Mod_{0,n}(\R)$ is given by
an ordering of $\{0,t_1,\ldots,t_\ell,1,\infty\}$ up to dihedral
permutation, we can identify it as above with an unoriented
$n$-sided polygon with edges indexed by the set
$\{0,t_1,\ldots,t_\ell,1,\infty\}$.  

\subsubsection{Cell-forms.}  A cell-form is a holomorphic differential
$\ell$-form on $\Mod_{0,n}$ with logarithmic singularities along the
boundary components of the stable compactification, having the
property that its singular locus forms the boundary of a single cell
in  the real moduli space $\Mod_{0,n}(\R)$.

Up to sign, the cell-form
diverging on a given cell is obtained by taking the successive
differences of the edges of the polygon representing that cell
(ignoring $\infty$) as factors in the denominator.  For example
the cell corresponding to the cyclic order $(0,1,t_1,t_3,\infty,t_2)$
is represented by the polygon on the left of the following figure,
and the cell-form diverging along it is given on the right:\vskip 1.5cm
$\displaystyle{\qquad\qquad\qquad\qquad \qquad\qquad\qquad
\longleftrightarrow \qquad\qquad \pm\,{{dt_1dt_2dt_3}\over
{(t_1-1)(t_3-t_1)(-t_2)}}}$

\vspace{-2cm}
\ \ \ \ \epsfxsize=3cm\epsfbox{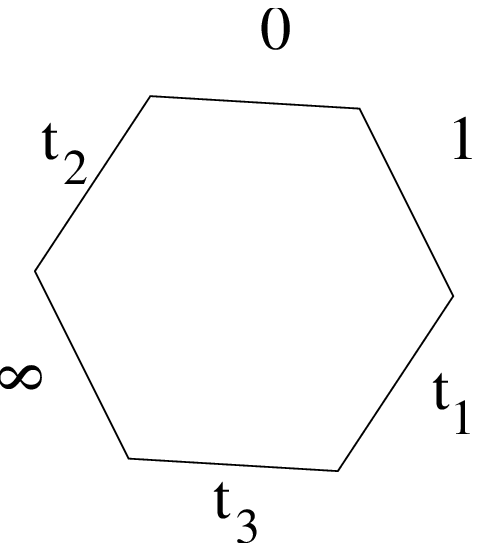}

\vspace{.1cm} 
Let ${\cal P}$ denote the $\Q$-vector space
generated by \emph{oriented} $n$-gons indexed by
$\{0,1,t_1,$ $\ldots,t_\ell,1,\infty\}$. The orientation fixes the sign
of the  corresponding cell form, and this  gives a map
\begin{equation}\label{introrhodef}
\rho:{\cal P}\rightarrow H^\ell(\Mod_{0,n}).\end{equation} In
proposition \ref{prop41} of section \ref{sec41} we prove that this map is surjective and
identify its kernel.  Chapter 3 is entirely devoted to a purely
combinatorial reformulation, in terms of polygons which simultaneously
represent both cells and cell-forms on moduli space, of the familiar
notions of convergence, divergence and residues of differential forms
along divisors.
\subsubsection{Cell-form cohomology basis.} 
We show that cell-forms provide a good framework for studying the logarithmic 
differential forms on $\Mod_{0,n}$, starting with the following result
(theorem \ref{thm01cellsspan}), whose proof is based
on Arnol'd's well-known construction of a different basis for the
cohomology group $H^\ell(\Mod_{0,n})$.

\vspace{.3cm}
\noindent {\bf Theorem.} {\it The set of $01$ cell-forms (those 
corresponding to polygons in which $0$ appears next to $1$ in the indexing 
of the edges) forms a basis for the cohomology group 
$H^\ell(\Mod_{0,n})$ of top-dimensional differential forms on the moduli space.}

\vspace{.3cm}
In particular, this shows that the cohomology group $H^\ell(\Mod_{0,n})$
is canonically isomorphic to the subspace of ${\cal P}$ of polygons having 
$0$ adjacent to $1$, providing a new approach.
\subsubsection{Insertion forms.} Insertion forms (definition \ref{definsforms})
are particular linear combinations of $01$ cell-forms having the 
property given in the following theorem (theorem \ref{insformsspan}), 
one of the main results of this paper.

\vspace{.3cm}
\noindent {\bf Theorem.} {\it The insertion forms 
form a basis for the space of top-dimensional logarithmic differential forms 
which converge on the closure of the standard cell of $\Mod_{0,n}(\R)$.}

\vspace{.3cm}In other words, insertion forms give a basis for the cohomology group 
$H^\ell(\Mod_{0,n}^\delta)$ of (classes of) forms having no poles along the 
boundary of the standard cell of $\Mod_{0,n}(\R)$, so that
the integral $(\ref{introint})$ converges, yielding a period.  

The insertion forms are defined in definition \ref{definsforms}, but the 
definition is based on the essential construction of Lyndon insertion words 
given in definition \ref{insw} and studied throughout section \ref{Lyndins}. 
The proof of this theorem uses all the polygon machinery developed in chapter 3.

\vspace{.3cm}

\subsubsection{Cell-zeta values.}  These are real numbers obtained by  
integrating insertion forms over the standard cell as in $(\ref{introint})$.
They are a generalization of multiple zeta values to a larger set of
periods on $\Mod_{0,n}$, such as
$$\int_{0<t_1<t_2<t_3<1} {{dt_1dt_2dt_3}\over{(1-t_1)(t_3-t_1)t_2}}.$$
Note that unlike the multiple zeta values, this is not an iterated integral
as in (\ref{itint}).
\subsubsection{Product map relations between cell-zeta values.} Via the pullback, the maps
$f:\Mod_{0,n}\rightarrow \Mod_{0,r}\times \Mod_{0,s}$ obtained by
forgetting disjoint complementary subsets of the marked points
$t_1,\ldots,t_\ell$ yield expressions for products of cell-zeta
values on $\Mod_{0,r}$ and $\Mod_{0,s}$ as linear combinations of
cell-zeta values on $\Mod_{0,n}$:
\begin{equation}\label{introcellprodmap}
\int_{X_1} \omega_1\int_{X_2}\omega_2=\int_{f^{-1}(X_1\times X_2)}
f^*(\omega_1\wedge\omega_2).\end{equation} There is a simple
combinatorial algorithm to compute the multiplication law in terms
of cell-forms. This is a geometric analog of the familiar quadratic
relations for multiple zeta values, and is  explained in section
\ref{prodmaps}.
\subsubsection{Dihedral relations between cell-zeta values} These relations between
cell-zeta values are given by \begin{equation}
\label{introdihedralrel}\int_X\omega=\int_X \sigma^*(\omega)\ ,
\end{equation} where $\sigma$ is an automorphism of $\Mod_{0,n}$ which
maps the standard cell to itself: $\sigma(X)=X$, and thus $\sigma$
is a dihedral permutation of the marked points $\{0,1,t_1,\ldots,
t_\ell, \infty\}$.

\subsubsection{The cell-zeta value algebra ${\cal C}$.} The multiplication
laws associated to product maps $(\ref{introcellprodmap})$ make the
space of all cell-zeta values on $\Mod_{0,n}$, $n\ge 5$, into a
$\Q$-algebra which we denote by ${\cal C}$. By Brown's theorem
\cite{Br2}, which states essentially that all periods on
$\Mod_{0,n}$ are linear combinations of multiple zeta values,
together with Kontsevitch's expression (\ref{itint}) of multiple
zeta values, we obtain the following result (theorem \ref{brownsthesis}).

\vspace{.3cm}
\noindent {\bf Theorem.} {\it The cell-zeta value algebra ${\cal C}$ is 
equal to the algebra of multiple zeta values ${\cal Z}$.}

\vspace{.1cm}
\subsubsection{The formal cell-zeta value algebra ${\cal FC}$.}
By lifting the previous constructions to the level of polygons along
the map $(\ref{introrhodef})$, we define in section \ref{cellalg} an
algebra of  formal cell-zeta values which we denote by ${\cal FC}$.
It is   generated by the {\it Lyndon insertion words} (see definition
\ref{insw}), which are formal
sums of polygons corresponding to the insertion forms introduced
above, subject to combinatorial versions of the product map
relations $(\ref{introcellprodmap})$ and the dihedral relations
$(\ref{introdihedralrel})$.  We consider this analogous
to the formal multizeta algebra ${\cal FZ}$, generated by formal symbols
representing convergent multiple zeta values, subject only to the 
convergent double shuffle and Hoffmann relations (\cite{Ho}).  The computer calculations 
in low weight described in chapter 4 motivated us to make the following 
conjecture, which essentially says that the product map and dihedral
relations (plus another simple family coming from combinatorial identities
on polygons, see definition \ref{threerels} for the complete definition of
the three families of relations) generate all relations between periods of
the moduli space.
 
\vspace {.3cm}
\noindent {\bf Conjecture.} {\it The formal cell-zeta algebra ${\cal FC}$
is isomorphic to the formal multizeta algebra ${\cal FZ}$.}

\vspace{.5cm} The paper is organized as follows. In $\S2$,  we
introduce cell forms and polygons and define the three familes of relations. 
In  $\S3$, we  define Lyndon insertion words of polygons, which may be of 
independent combinatorial interest.  These are used to construct the insertion 
basis of convergent forms in $\S4$. In $\S\ref{calculations}$, we give 
complete computations of  this basis and the  corresponding product map
relations for $\Mod_{0,n}$, where $n=5,6,7$.

In the remainder of this introduction we sketch the connections
between the formal cell-zeta value algebra and standard results and
conjectures in the theory of multiple zeta values and mixed Tate
motives.

\vspace{0.2cm}
\subsection{Relation to mixed Tate motives and conjectures}
Let  $\MT(\Z)$ denote the category of mixed Tate motives which are
unramified over $\Z$ \cite{D-G}.  Let  $\delta$ denote the standard
cyclic structure on $S=\{1,\ldots,n\}$, and let $B_{\delta}$ denote
the divisor which bounds the standard cell $X$. Let
$A_\delta$ denote the set of all remaining divisors on
$\overline{\Mod}_{0,S}\backslash \Mod_{0,S}$, so that
$\Mod_{0,S}^{\delta}=\overline\Mod_{0,S}\setminus A_{\delta}$ (\cite{Br2}).
We write:
\begin{equation} \label{Mdelta}
M_\delta= H^\ell(\overline{\Mod}_{0,n}\backslash A_\delta, B_\delta
\backslash (B_\delta \cap A_\delta))\ .
\end{equation}
By a result due to Goncharov and Manin  \cite{G-M}, $M_\delta$
defines an element in $\MT(\Z)$, and therefore  is equipped with an
increasing weight filtration $W$.  They show that $\gr^W_\ell
M_\delta$ is isomorphic to the de Rham cohomology
$H^\ell(\Mod^\delta_{0,n})$, and that $\gr^W_0 M_\delta$ is
isomorphic to the dual of the relative Betti homology
$H_\ell(\overline{\Mod}_{0,n}, B_\delta )$.

Let $M$ be any element in  $\MT(\Z)$. A framing for $M$ consists of
an integer $n$ and non-zero maps
\begin{equation}
v  \in  \Hom( \Q(-n), \gr^W_{2n} M) \quad \hbox{ and } \quad f\in
 \Hom( \gr^W_{0} M, \Q(0)) \ .
\end{equation}
Two framed motives $(M,v,f)$ and $(M',v',f')$ are said to be
equivalent if there is a morphism $\phi:M\rightarrow M'$ such that
$\phi\circ v= v'$ and $f'\circ\phi =f$. This generates an
equivalence relation whose equivalence classes are denoted
$[M,v,f]$. Let $\FMT(\Z)$ denote the set of equivalence classes
of framed mixed Tate motives which are unramified over $\Z$, as
defined in \cite{Go1}. It is a commutative, graded Hopf algebra over $\Q$.

To every convergent cohomology class $\omega \in
H^\ell(\Mod^\delta_{0,n})$, we associate the following $\ell$-framed mixed
Tate motive:
\begin{equation} \label{momegadef}
m(\omega) = \big[M_\delta, \omega,[X]\big]\ ,
\end{equation}
where $[X]$ denotes the relative homology class of the
standard cell. This defines a map ${\cal FC} \rightarrow \FMT(\Z)$.
 The maximal period of $m(\omega)$ is exactly the
cell-zeta value
$$\int_{X} \omega\ .$$

\begin{prop} The dihedral symmetry relation and product map relations
are motivic. In other words,
\begin{eqnarray}
m(\sigma^*(\omega))&=& m(\omega) \nonumber \ ,\\
m(\omega_1\cdot\omega_2) & = & m(\omega_1)\otimes m(\omega_2)\ , \nonumber
\end{eqnarray}
for every dihedral symmetry $\sigma$ of $X$, and for every
modular shuffle product $\omega_1\cdot\omega_2$ of convergent forms
$\omega_1,\omega_2$ on $\Mod_{0,r}$, $\Mod_{0,s}$ respectively.
\end{prop}

The motivic nature of our constructions will be clear from the
definitions. We therefore obtain a well-defined map $m$ from the
algebra of formal cell-zeta numbers ${\cal FC}$ to $\FMT(\Z)$.  
On $\Mod_{0,5}$, there is a unique element $\zeta_2\in {\cal FC}$ whose period 
is $\zeta(2)$, which maps to $0$ in $\FMT(\Z)$.

\begin{conj} ${\cal FC}$ is a free $\Q[\zeta_2]$-module, and the induced map
$$m:\mathcal{FC}/\zeta_2{\cal FC}\To \widehat\FMT(\Z)$$  
is an isomorphism.
\end{conj}

Since the structure of $\FMT(\Z)$ is known, we are led to more precise
conjectures on the structure of the formal cell-zeta algebra. To
motivate this, let $\Lie=\Q[e_3,e_5,\ldots, ]$ denote the free Lie
algebra generated by one element $e_{2n+1}$ in each odd degree. Set
$$\F = \Q[e_2] \oplus \Lie\ .   $$
The underlying graded vector space is generated by, in increasing
weight:
$$e_2 \spc e_3 \spc e_5 \spc e_7 \spc [e_3,e_5] \spc e_9 \spc [e_3,
e_7] \spc [e_3,[e_5,e_3]]\,
, \, e_{11} \spc [e_3,e_9]\, , \, [e_5,e_7] \spc \ldots \ .$$
Let  ${\cal U}\F$  denote the universal enveloping algebra of the
Lie algebra ${\F}$. Then, setting $\widehat\FMT(\Z)=\FMT(\Z)\otimes_\Q
\Q[\zeta_2]$, it is known that $\widehat\FMT(\Z)$ is dual to
${\cal U}\F$. From the explicit description of $\F$ given above, one
can deduce that the graded dimensions $d_k= \dim_\Q \gr^W_k
\widehat\FMT(\Z)$ satisfy Zagier's recurrence relation
\begin{equation} \label{ZagierRec}
d_k=d_{k-2}+d_{k-3}\ ,
\end{equation}
with the initial conditions $d_0=1$, $d_1=0$, $d_2=1$.


\begin{conj} The dimension of the $\Q$-vector space of formal cell-zeta
values on $\Mod_{0,n}$, modulo all linear relations obtained from
the dihedral and modular shuffle relations, is equal to $d_\ell$,
where $n=\ell+3$. 
\end{conj}

 We verified this  conjecture for $\Mod_{0,n}$ for
$n\leq 9$ by direct calculation (see $\S\ref{calculations}$).  When
$n=9$, the dimension of the convergent cohomology
$H^6(\Mod_{0,9}^\delta)$ is 1089, and after taking into account all
linear relations coming from dihedral and modular shuffle products,
this reduces to a vector space of dimension $d_6=2$.

To compare this picture with the classical picture of multiple zeta
values, let ${\cal FZ}$  denote the formal multizeta algebra. This
is the quotient of the free $\Q$-algebra generated by formal symbols
$(\ref{itint})$ modulo the regularised double shuffle relations. It
has been  conjectured that ${\cal FZ}$ is isomorphic to $\widehat\FMT(\Z)$,
and proved (cf. \cite{Te1}) that the dimensions $d_\ell$ are actually upper
bounds for the dimensions of the weight $\ell$ parts  of ${\cal FZ}$.  This
leads us to the second main conjecture.

\begin{conj}\label{wewish} The formal algebras ${\cal FC}$ and ${\cal FZ}$ are
isomorphic.
\end{conj}
Put more prosaically, this states that
the formal ring of periods of
$\Mod_{0,n}$ modulo dihedral and modular shuffle relations, is
isomorphic to the formal ring of periods of the motivic fundamental
group of $\Mod_{0,4}$ modulo the regularised double shuffle
relations.

By (\ref{itint}), we have a natural linear map ${\cal FZ}\rightarrow
{\cal FC}$. However, at present we cannot show that it is an algebra
homomorphism. Indeed, although it is easy to deduce the regularised
shuffle relation for the image of ${\cal FZ}$ in ${\cal FC}$ from
the dihedral and modular shuffle relations, we are unable to deduce
the regularised stuffle relations.  For further detail on this question,
see remark \ref{laterremark} below.
\begin{rem}
The motivic nature of the regularised double shuffle relations
proved to be somewhat difficult to establish \cite{Go1}, \cite{Go2}, \cite{Te1}. It
is interesting that the motivic nature of the dihedral and modular
shuffle relations we define here is immediate.
\end{rem}

\vspace{.3cm}
\section{The cell-zeta value algebra associated to moduli spaces of curves}

Let $\Mod_{0,n}$, $n\geq 4$ denote the moduli space of genus zero
curves (Riemann
spheres) with $n$ ordered marked points $(z_1,\ldots,z_n)$.  This
space is described by the set of $n$-tuples of distinct points
$(z_1,\ldots,z_n)$ modulo the equivalence relation given by the
action of ${\PSL}_2$.  Because this action is triply transitive,
there is a unique representative of each equivalence class such that
$z_1=0$, $z_{n-1}=1$, $z_n=\infty$.  We define simplicial coordinates
$t_1,\ldots,t_\ell$ on $\Mod_{0,n}$ by setting
\begin{equation} \label{zsimp}
t_1=z_2 \ ,\quad  t_2=z_3 \ , \quad \ldots\ ,\quad  t_\ell=z_{n-2},
\end{equation}
where $\ell=n-3$ is the dimension of $\Mod_{0,n}(\C)$.
This gives the familiar identification
\begin{equation}\label{simplicialisom}
\Mod_{0,n} \cong \{ (t_1,\ldots, t_\ell) \in
(\Pro^1-\{0,1,\infty\})^\ell\mid
t_i\neq t_j \hbox{ for all } i\neq j \}\
.\end{equation}

\vspace{.3cm}
\subsection{Cell forms}

\begin{defn} Let $S=\{1,\ldots,n\}$.  A {\it cyclic structure}
  $\gamma$ on $S$ is
a cyclic ordering  of the elements of $S$ or equivalently, an
identification of the elements of $S$ with the edges of an oriented
$n$-gon modulo rotations. A {\it dihedral structure} $\delta$ on $S$
is an identification with the edges of an unoriented $n$-gon modulo
dihedral symmetries.
\end{defn}
We can write a cyclic structure as an ordered $n$-tuple
$\gamma=(\gamma(1), \gamma(2),...,\gamma(n))$ considered up to
cyclic rotations.



\begin{defn}\label{cellform} Let
  $(z_1,\ldots,z_n)=(0,t_1,\ldots,t_\ell,1,\infty)$
be a representative of a point on $\Mod_{0,n}$ as above. Let
$\gamma$ be a cyclic structure on $S$, and let $\sigma$ be the
unique ordering of $z_1,\ldots, z_n$ compatible with $\gamma$ such
that $\sigma(n)=n$. The \emph{cell-form} corresponding to $\gamma$
is defined to be the differential $\ell$-form
\begin{equation}\label{omegadef}
\omega_{\gamma} = [z_{\sigma(1)},z_{\sigma(2)},\ldots,z_{\sigma(n)}]=
\frac{dt_1\cdots dt_\ell}
{(z_{\sigma(2)}-z_{\sigma(1)}) (z_{\sigma(3)}-z_{\sigma(2)})\cdots
  (z_{\sigma(n-1)} -z_{\sigma(n-2)})}.
\end{equation}
In other words, by writing the terms of $\omega_\gamma =
[z_{\sigma(1)}, ... ,z_{\sigma(n)}]$ clockwise around a polygon, the
denominator of a cell form is just
the product of
successive differences $(z_{\sigma(i)} -z_{\sigma(i-1)})$
 with the two
factors containing $\infty$ simply left out.
\end{defn}

\begin{rem}\label{grlem}
To every dihedral structure there correspond two opposite cyclic
structures. If these are given by $\gamma$ and $\tau$, then we have
\begin{equation} \omega_{\gamma} = (-1)^n
  \omega_{\tau}. \end{equation}
\end{rem}

\begin{example} Let $n=7$, and $S=\{1,\ldots, 7\}$. Consider the cyclic
structure $\gamma$ on $S$ given by the order $1635724$.  The unique
ordering $\sigma$ of $S$ compatible with $\gamma$ and having
$\sigma(n)=n$, is the ordering $2416357$, which can be depicted by
writing the elements $z_{\sigma(1)},\ldots,z_{\sigma(7)}$, or $0$,
$1$, $t_2$, $t_4$, $\infty$, $t_1$, $t_3$ clockwise around a circle:
$$\gamma = (z_{\sigma(1)},\ldots,z_{\sigma(7)})=
(t_1,t_3,0,1,t_2,t_4,\infty).$$
The corresponding cell-form on $\Mod_{0,7}$ is
$$\omega_{\gamma}=[t_1,t_3,0,1,t_2,t_4,\infty]= {dt_1 dt_2 dt_3 dt_4
  \over (t_3-t_1) (-t_3)(t_2-1)(t_4-t_2)}\ .$$
\end{example}

The symmetric group $\Sym(S)$ acts on $\Mod_{0,n}$ by permutation of
the marked points. It therefore acts both on the set of cyclic
structures $\gamma$, and also on the ring of differential forms on
$\Mod_{0,n}$. These actions coincide for cell forms.

 For any cyclic structure $\gamma$ on $S$, let $
D_{\gamma}\subset \Sym(S)$ denote the group of automorphisms of the
dihedral structure which underlies $\gamma$, which is a dihedral
group of order $2n$. 

\begin{lem} \label{lemsymaction}
For every cyclic structure $\gamma$ on $S$, we have the formula:
\begin{equation}\label{symaction}
\sigma^*(\omega_\gamma) = \omega_{\sigma(\gamma)}\qquad \hbox{ for
all } \sigma \in \Sym(S)\ .\end{equation}
\end{lem}
\begin{proof}Consider the logarithmic $n$-form on
$(\Pro^1)^S_*$ defined by the formula:
\begin{equation}\label{omegalift}
\widetilde{\omega}_\gamma = {dz_1 \wedge \ldots \wedge dz_n \over
(z_{\gamma(1)}-z_{\gamma(2)})\ldots (z_{\gamma(n)}- z_{\gamma(1)})}
\ .\end{equation}
 It clearly satisfies
 $\sigma^*(\widetilde{\omega}_{\gamma})=
 \widetilde{\omega}_{\sigma(\gamma)}$ for all $\sigma \in D_{\gamma}$.
 A simple calculation shows that $\widetilde{\omega}_\gamma$
is invariant under the action of $\PSL_2$ by M\"obius
transformations. Let $\pi:(\Pro^1)^S_* \rightarrow \Mod_{0,S}$
denote the projection map with fibres isomorphic to  $\PSL_2$. There
is a unique (up to scalar multiple in $\Q^\times$) non-zero
invariant logarithmic 3-form $v$ on $\PSL_2(\C)$ which is defined over
$\Q$. Then, by renormalising $v$ if necessary, we have
$\omega_{\gamma}\wedge v = \widetilde{\omega}_{\gamma}\ .$
 In fact,
 $\omega_{\gamma}$ is the unique
$\ell$-form on $\Mod_{0,S}$ satisfying  this equation. We deduce
that $\sigma^*(\omega_{\gamma})= \omega_{\sigma(\gamma)}$ for all
$\sigma \in D_{\gamma}$.
\end{proof}


Each dihedral structure $\eta$ on $S$ corresponds to a unique
connected component of the real locus $\Mod_{0,n}(\R)$, namely the
component associated to the set of Riemann spheres with real marked
points $(z_1,\ldots,z_n)$ whose real ordering is given by $\eta$. We
denote this component by $X_{S,\eta}$ or $X_{n,\eta}$.  It is an
algebraic manifold with corners with the combinatorial structure of
a Stasheff polytope, so we often refer to it as a \emph{cell}. A
cyclic structure compatible with $\eta$ corresponds to a choice of
orientation of this cell.
 
\begin{defn}Let $\delta$ once and for all denote the
cyclic order corresponding to the ordering $(1,2,\ldots,n)$. We call
$X_{S,\delta}=X_{n,\delta}$ the \emph{standard cell}.  It is the set
of points on $\Mod_{0,n}$ given by real marked points
$(0,t_1,\ldots,t_\ell,1,\infty)$ in that cyclic order; in simplicial
coordinates it is given by  the standard real simplex
$0<t_1<\ldots<t_\ell<1$.
\end{defn}

The distinguishing feature of cell-forms, from which they derive their name,
is given in the following proposition.

\vspace{.2cm}
\begin{prop}\label{CORomegapoles} Let $\eta$ be a dihedral structure
on $S$, and let $\gamma$ be either of the two cyclic substructures
of $\eta$.  Then the cell form $\omega_\gamma$  has  simple poles along the
boundary of the cell $X_{S,\eta}$ and no poles anywhere else.
\end{prop}

\begin{proof}
Let $D\subset \overline{\Mod}_{0,S}\backslash \Mod_{0,S}$ be a
divisor given by a partition $S=S_1\coprod S_2$ such that 
$|S_i|>1$ for $i=1,2$.  In \cite{Br2}, the following
notation was introduced:
$$\I_D(i,j) = \I(\{i,j\}\subset S_1) + \I(\{i,j\}\subset S_2)\ ,$$
where $\I(A\subset B)$ is the indicator function which takes the
value $1$ if $A$ is contained in $B$ and $0$ otherwise. Therefore
$\I_D(i,j)\in  \{0,1\}. $  Then we have
\begin{equation} \label{ordcell}
2\,\ord_D (\omega_\gamma) = (\ell-1)- \I_D(\gamma(1),\gamma(2)) -
\I_D(\gamma(2),\gamma(3)) - \ldots- \I_D(\gamma(n),\gamma(1))\ .
\end{equation}
To prove this, observe that $\omega_\gamma=f_\gamma \omega_0$, where
$$f_\gamma= \prod_{i\in \Z/n\Z} {(z_i-z_{i+2}) \over (z_{\gamma(i)}-z_{\gamma(i+1)})}
\ ,$$ and
$$\omega_0 = {dt_1\ldots dt_\ell \over t_2(t_3-t_1)(t_4-t_2)\ldots
  (t_\ell-t_{\ell-2}) (1-t_{\ell})}$$
is the canonical volume form with no zeros or poles along the
standard cell  defined in \cite{Br2}. The proof of (\ref{ordcell})
follows on applying proposition 7.5 from \cite{Br2}.

Now, $(\ref{ordcell})$ shows that $\omega_\gamma$ has the worst singularities
when the most possible $\I_D(\gamma(i),\gamma(i+1))$ are equal to 1.  This
happens when only two of them are equal to zero, namely
$$S_1=\{\gamma(1),\gamma(2),\ldots,\gamma(k)\}\quad \hbox{ and } \quad
S_2= \{\gamma(k+1),\gamma(k+2),\ldots,\gamma(n)\},\ \ 2\le k\le n-2.$$
In this case, $(\ref{ordcell})$ yields $2\ord_D\omega_\gamma=(\ell-1)-(n-2)=
-2$, so $\ord_D \omega_{\gamma}=-1$. In all other cases we must therefore
have $\ord_D \omega_{\gamma} \geq 0$. Thus the  singular locus of
$\omega_{\gamma}$ is precisely given by the set of divisors bounding
the cell $X_{S,\eta}$.
\end{proof}

\vspace{.3cm}
\subsection{01 cell-forms and a basis of the cohomology of
  $\Mod_{0,n}$}
We first derive some useful identities between certain rational
functions. Let $S=\{1,\ldots, n\}$ and let  $v_1,\ldots,v_n$ denote
coordinates on $\A^n$.
 For every cyclic
structure $\gamma$ on $S$, let $\cfl \gamma \cfr=\cfl
v_{\gamma(1)},\ldots,v_{\gamma(n)} \cfr$ denote the rational
function
\begin{equation}\label{cellfuncdef}
{{1}\over{(v_{\gamma(2)}-v_{\gamma(1)})\cdots
(v_{\gamma(n)}-v_{\gamma(n-1)})(v_{\gamma(1)}-v_{\gamma(n)})}}\in
\Z\Big[v_i, {1\over v_i-v_j}\Big]\ .\end{equation} We refer to such
a function as a cell-function.
 We can
extend its definition linearly to $\Q$-linear combinations of cyclic
structures. Let $X=\{x_1,\ldots, x_n\}$ denote any alphabet on $n$
symbols. Recall that the shuffle product \cite{Re} is defined on
linear combinations of words on $X$ by the inductive formula
\begin{equation}\label{shufflerec}
w \sha e = e\sha w \quad \hbox{ and }\quad   a w \sha a'w' = a(w\sha
a'w') + a'(aw\sha w')\ , \end{equation} where $w,w'$ are any words
in $X$ and $e$ denotes the empty or trivial word.

\begin{defn}\label{defshufprod}  Let $A, B\subset S$  such that
$A\cap B=C=\{c_1,\ldots,c_r\}$ with $r\ge 1$.  Let $\gamma_A$ be a cyclic order on
$A$ such that the elements $c_1,\ldots,c_r$ appear in their standard
cyclic order, and let $\gamma_B$ be a cyclic order on $B$ with the
same property.  We write
$\gamma_A=(c_1,A_{1,2},c_2,A_{2,3},\ldots,c_r,A_{r,1})$ and
$\gamma_B=(c_1,B_{1,2},c_2,B_{2,3},\ldots,c_r,B_{r,1})$, where the
$A_{i,i+1}$,  (resp. the $B_{i,i+1}$) together with $C$, form a
partition of $A$ (resp. $B$). We denote the {\it shuffle
  product} of
the two cell-functions $\cfl \gamma_A \cfr$ and $\cfl \gamma_B \cfr
$ with respect to $c_1,\ldots,c_r$ by
$$\cfl \gamma_A \cfr \sha_{c_1,\ldots,c_r}
\cfl\gamma_B\cfr$$
which is defined to be the sum of cell functions
\begin{equation}\label{formalshufprod}\cfl c_1,A_{1,2}\sha
  B_{1,2},c_2,A_{2,3}\sha B_{2,3},
\ldots,c_r,A_{r,1}\sha B_{r,1}\cfr \ .\end{equation}
\end{defn}

The shuffle product of two cell-functions is related to their actual
product by the following lemma.

\begin{prop}\label{shufprod} Let $A,B \subset S$, such that $|A\cap
B|\geq 2$.  Let $\gamma_A$, $\gamma_B$ be cyclic structures on $A,B$
such that the  cyclic structures on $A\cap B$ induced by  $\gamma_A$
and $\gamma_B$ coincide. If $\gamma_{A\cap B}$ denotes the induced
cyclic structure on $A\cap B$, we have:
\begin{equation}\label{cellfunctionshufflerel}
{{\cfl\gamma_A\cfr \cdot \cfl \gamma_B\cfr }\over{ \cfl
\gamma_{A\cap B}\cfr }}= \cfl \gamma_A\cfr \sha_{\gamma_{A\cap B}}
\cfl \gamma_B\cfr \ .
\end{equation}
\end{prop}

\begin{proof}
Write the cell functions $\cfl \gamma_A \cfr $ and $\cfl \gamma_B
\cfr$ as $\cfl a_{i_1}, P_1, a_{i_2}, P_2,\ldots ,a_{i_r}, P_{r}\cfr
$ and\break $\cfl a_{i_1}, R_1, a_{i_2}, R_2, \ldots, a_{i_r}
,R_{r}\cfr$, where $P_i, R_i$ for $1\leq i\leq r$ are tuples of
elements in $S$. Let $\Delta_{ab}= (b-a)$. We will first prove the
result for $r=2$ and $P_2,R_2=\emptyset$:
\begin{equation}\label{r=2}
\Delta_{ab}\Delta_{ba}\cfl a ,p_1,\ldots ,p_{k_1} ,b\cfr  \cfl a
,r_1,\ldots,
  r_{k_2},b\cfr   =\cfl a ,(p_1,\ldots ,p_{k_1})\sha (r_1,\ldots, r_{k_2}),b\cfr  .
\end{equation}
We prove this case by induction on $k_1+k_2$.
Trivially, for $k_1+k_2=0$ we have
\begin{equation*}
\Delta_{ab}\Delta_{ba}\cfl a, b\cfr  \cfl a, b\cfr   = \cfl a, b\cfr
.
\end{equation*}
Now assume the induction hypothesis that
\begin{align*} & \Delta_{ab}\Delta_{ba}\cfl a,
  p_2 ,\ldots, p_{k_1},
  b\cfr  \cfl a ,r_1,\ldots ,r_{k_2}, b\cfr   = \cfl a, \bigl( (p_2,\ldots, p_{k_1})\sha
  (r_1, \ldots
  ,r_{k_2})\bigr),b\cfr
\hbox{ and } \\ & \Delta_{ab}\Delta_{ba}\cfl a ,p_1,\ldots, p_{k_1},
    b\cfr  \cfl a, r_2,\ldots
  ,r_{k_2},b\cfr   = \cfl a ,\bigl( (p_1,\ldots ,p_{k_1})\sha (r_2,\ldots
  ,r_{k_2})\bigr),b\cfr  .\end{align*}
To lighten the notation, let $p_2,\ldots ,p_{k_1}=\underline{p}$ and
$r_2,\ldots, r_{k_2}=\underline{r}$. By the  shuffle recurrence
formula $(\ref{shufflerec})$ and  the induction hypothesis:
\begin{align*}
\cfl a, \bigl( (p_1,\underline{p})\sha (r_1,\underline{r})\bigr) ,
b\cfr   & = \cfl a ,p_1, \bigl( (\underline{p} )\sha
(r_1,\underline{r})\bigr) , b\cfr   +
\cfl a ,r_1 ,\bigl( (p_1, \underline{p}) \sha (\underline{r})\bigr), b\cfr   \\
&= \frac{\Delta_{p_1b}\cfl p_1, \bigl( (\underline{p})
    \sha (r_1, \underline{r})\bigr), b\cfr  }{\Delta_{ab}\Delta_{ap_1}} +
  \frac{ \Delta_{r_1b}\cfl r_1, \bigl( (p_1,\underline{p})
    \sha \underline{r}) \bigr), b\cfr  } {\Delta_{ab}\Delta_{ar_1}}\\
&=\frac{\Delta_{p_1b}\Delta_{bp_1}\Delta_{p_1b}\cfl p_1
,\underline{p}, b\cfr  \cfl
  p_1 , r_1, \underline{r}, b\cfr  }{\Delta_{ab}\Delta_{ap_1} }+
\frac{\Delta_{r_1b}\Delta_{br_1}\Delta_{r _1b}\cfl r_1 ,p_1,
  \underline{p},
  b\cfr  \cfl  r_1 , \underline{r}, b\cfr  }{\Delta_{ab}\Delta_{ar_1} }
\end{align*}
Using identities such as $\cfl p_1 ,\underline{p},b\cfr
={\Delta_{ap_1}\Delta_{ba} \over \Delta_{bp_1}}\cfl a,
p_1,\underline{p}, b\cfr$, this is
$$\Big[{\Delta^2_{p_1b}\Delta_{bp_1} \over \Delta_{ab}
\Delta_{ap_1}}\, {\Delta_{ap_1}\Delta_{ba}\over \Delta_{bp_1}}
\,{\Delta_{ba}\Delta_{ar_1} \over \Delta_{bp_1}\Delta_{p_1r_1}} +
{\Delta^2_{r_1 b} \Delta_{br_1}  \over \Delta_{ab}\Delta_{ar_1}}
{\Delta_{ap_1}\Delta_{ba} \over \Delta_{r_1p_1}\Delta_{br_1} }
{\Delta_{ba}\Delta_{ar_1} \over \Delta_{br_1}}\Big]\cfl a , p_1,
\underline{p}, b\cfr \cfl a, r_1, \underline{r}, b\cfr  $$
$$=\Delta_{ab} \Big[{\Delta_{ar_1}\Delta_{bp_1} \over \Delta_{p_1r_1}} + {\Delta_{br_1}\Delta_{ap_1}\over \Delta_{r_1p_1}}\Big]\cfl a , p_1, \underline{p},  b\cfr  \cfl
a, r_1, \underline{r}, b\cfr = \Delta_{ab}\Delta_{ba}\cfl a , p_1,
\underline{p},  b\cfr  \cfl a, r_1, \underline{r}, b\cfr .$$
 The last equality is the Pl\"ucker relation $\Delta_{ar_1}\Delta_{bp_1}- \Delta_{br_1}\Delta_{ap_1}=\Delta_{p_1r_1}\Delta_{ba}$. This proves the identity \eqref{r=2}.
Now,  using the identity
\begin{align*}
\cfl a_{i_1} P_1 a_{i_2} P_2 \ldots a_{i_r} P_r \cfr  =
\Delta_{a_{i_2}a_{i_1}} \cfl a_{i_1} P_1 a_{i_2}\cfr   \times
\Delta_{a_{i_3}a_{i_2}} \cfl a_{i_2} P_2 a_{i_3}\cfr  \times \cdots
\times \Delta_{a_{i_r} a_{i_1}} \cfl a_{i_r} P_r a_{i_1}\cfr  ,
\end{align*}
the general case follows   from \eqref{r=2}.
\end{proof}

\begin{cor} Let $X$ and $Y$ be disjoint sequences of indeterminates and
let $e$ be an indeterminate not appearing in either $X$ or $Y$.  We
have the following identity on cell functions:
\begin{equation} \label{cellfunction1shuffsare0}
\cfl (X,e)\sha_e (Y,e) \cfr= \cfl X\sha Y ,e\cfr =0.
\end{equation}
\end{cor}

\begin{proof} Write $X=x_1,x_2,...,x_n$ and $Y=y_1,y_2,...,y_m$.
By the recurrence formula for the shuffle product and proposition
\ref{shufprod}, we have
\begin{align*} \cfl X\sha Y ,e\cfr   & = \cfl x_1, (x_2,...,x_n\sha y_1,...,y_m) , e\cfr   +
\cfl y_1, (x_1,...,x_n \sha y_2,...,y_m) , e\cfr   \\
&= \cfl X,e\cfr  \cfl x_1, Y, e\cfr  (e-x_1)(x_1-e) + \cfl y_1, X,
e\cfr   \cfl Y,e\cfr
(y_1-e)(e-y_1) \\
&= \frac{(e-x_1)(x_1-e)}{(x_2-x_1)\cdots (e-x_n)
   (x_1-e)\ (y_1-x_1)(y_2-y_1)\cdots (e-y_m)(x_1-e)} \\
&\qquad + \frac{(y_1-e)(e-y_1)} {
   (x_1-y_1)(x_2-x_1)\cdots (e-x_n) (y_1-e) \ (y_2-y_1)\cdots
   (e-y_m)(y_1-e)} \\
&= \frac{ (-1) + (-1)^2 }{ (x_2-x_1) \cdots (e-x_n)
   \ (y_1-x_1)(y_2-y_1)\cdots (e-y_m)} =0 \ .\\
\end{align*}
\end{proof}
\noindent By specialization, we  can formally extend the definition
of a cell function to the case where some of the terms $v_i$ are
constant, or one of the  $v_i$ is infinite, by setting
$$\cfl v_1,\ldots, v_{i-1}, \infty, v_{i+1},\ldots, v_n\cfr =
\lim_{x\rightarrow \infty} x^2 \cfl v_1,\ldots,
v_{i-1},x,v_{i+1},\ldots, v_n\cfr$$ $$= {1 \over (v_2-v_1)\ldots
(v_{i-1}-v_{i-2})(v_{i+2}-v_{i+1}) \ldots(v_{n}-v_{n-1})(v_1-v_n) }\
.$$ \noindent This is the rational function obtained by omitting all
terms containing $\infty$. By taking the appropriate limit, it is
clear that $(\ref{cellfunctionshufflerel})$ and
$(\ref{cellfunction1shuffsare0})$ are valid in this case too. In the
case where $\{v_1,\ldots, v_n\} = \{0,1,t_1,\ldots, t_\ell,
\infty\}$ we have the formula
\begin{equation}
[v_1,\ldots, v_n]= \cfl v_1,\ldots, v_n \cfr \, dt_1dt_2\ldots
dt_\ell\ .
\end{equation}

\vspace{.1cm}
\begin{defn} A {\it $01$ cyclic (resp. dihedral) structure} is a cyclic 
(resp. dihedral)
structure on $S$ in which the numbers $1$ and $n-1$ are consecutive.
Since $z_1=0$ and $z_{n-1}=1$, a $01$ cyclic (or dihedral) structure
is a set of orderings of the set
$\{z_1,\ldots,z_n\}=\{0,t_1,\ldots,t_\ell,1, \infty\}$, in which the
elements $0$ and $1$ are consecutive.  In these terms, each dihedral
structure can be written as an ordering $(0,1,\pi)$ where $\pi$ is
some ordering of $\{t_1,\ldots,t_\ell,\infty\}$.  To each such
ordering we associate a cell-function $\cfl 0,1,\pi \cfr$, which is
called a $01$ cell-function.
\end{defn}
Since $01$ cell-functions corresponding to different $\pi$ are
clearly different, it follows that there exist exactly $(n-2)!$
distinct $01$ cell-functions $\cfl 0,1,\pi \cfr$.  To these
correspond $(n-2)!$ distinct $01$ cell-forms
$\omega_{(0,1,\pi)}=\cfl 0,1,\pi\cfr \, dt_1\ldots dt_\ell$.

\begin{thm} \label{thm01cellsspan} The set of $01$ cell-forms
$\omega_{(0,1,\pi)}$, where $\pi$ denotes any ordering of $\{t_1,\ldots,t_\ell,
\infty\}$, has cardinal $(n-2)!$ and forms a basis of $H^\ell(\Mod_{0,n},\Q)$.
\end{thm}
\begin{proof} The proof is based on the following well-known result due to
Arnol'd \cite{Ar}.
\begin{thm} 
A basis of $H^\ell(\Mod_{0,n},\Q)$ is given by the classes of the forms
\begin{equation}\label{Omegadefn}
\Omega(\underline{\varepsilon}) := {dt_1 \ldots dt_\ell \over
(t_1-\varepsilon_1)\ldots (t_\ell-\varepsilon_\ell)}\ ,\quad
\varepsilon_{i} \in E_i\ ,\end{equation} where $E_1=\{0,1\}$ and
$E_i =\{0,1,t_1,\ldots, t_{i-1}\}$ for $2\leq i\leq \ell$.
\end{thm}
It suffices to prove that each element
$\Omega(\underline{\varepsilon})$ in $(\ref{Omegadefn})$ can be
written as a linear combination of $01$ cell-forms.  We begin by
expressing a given rational function
${{1}\over{(t_1-\varepsilon_1)\cdots (t_\ell-\epsilon_\ell)}}$ as a
product of cell-functions and then apply proposition \ref{shufprod}.
 To every $t_i$, we associate its {\it type} $\tau(t_i)\in\{0,1\}$
 (which depends on $\varepsilon_1,\ldots, \varepsilon_\ell$)
as follows. If $\varepsilon_i=0$ then $\tau(t_i)=0$; if
$\varepsilon_i=1$, then $\tau(t_i)=1$, but if $\varepsilon_i\ne 0,1$
then $\varepsilon_i=t_j$ for some $j<i$, and the type of $t_i$ is
defined to be equal to the type of $t_j$. Since the indices
decrease,
 the type is
well-defined.

We associate a cell-function $F_i$ to each factor
$(t_i-\varepsilon_i)$ in the denominator of
$\Omega(\underline{\varepsilon})$ as follows:
\begin{equation} F_i=
\begin{cases}\ \ \cfl  0,1,t_i,\infty\cfr  &\hbox{if}\ \varepsilon_i=1\\
-\cfl  0,1,\infty,t_i\cfr  &\hbox{if}\ \varepsilon_i=0\\
\ \ \cfl  0,1,\varepsilon_i,t_i,\infty\cfr  &\hbox{if}\
\varepsilon_i\ne 1\ \hbox{and
the type}\ \tau(t_i)=1\\
-\cfl  0,1,\infty,t_i,\varepsilon_i\cfr  &\hbox{if }\varepsilon_i\ne
0\ \hbox{and the type}\ \tau(t_i)=0\ .
\end{cases}
\end{equation}
We have
$$\Omega(\underline{\varepsilon})=\Delta\prod_{i=1}^\ell F_i\ ,$$
where
$$\Delta=\prod_{j|\varepsilon_j\ne 0,1}(-1)^{\tau(\varepsilon_j)-1}
(\varepsilon_j-\tau(\varepsilon_j))$$
is exactly the factor occurring when multiplying cell-functions as in
proposition
\ref{shufprod}.  This product can be expressed as a shuffle
product, which is a sum of
cell-functions.  Furthermore each one corresponds to a cell beginning
$0,1,\ldots$ since this is the case for all of the $F_i$.
The $01$-cell forms thus span $H^\ell(\Mod_{0,n},\Q)$.
Since there are exactly $(n-2)!$ of them, and
since $\dim H^\ell(\Mod_{0,n},\Q)=(n-2)!$, they must form a basis.
\end{proof}

\vspace{.3cm}
\subsection{Pairs of polygons and multiplication}\label{prodmapsection}

\begin{defn}\label{polspace} 
Let $S=\{1,\ldots,n\}$, and let ${\cal P}_S$ denote the 
$\Q$-vector space generated by the set of cyclic structures $\gamma$ on $S$,
i.e. by planar polygons with $n$ sides indexed by $S$.  Let
$\tilde{\cal P}_S$ denote the $\Q$-vector
space generated by the set of cyclic structures $\gamma$ on $S$,
modulo the relation $\gamma=(-1)^n\overleftarrow{\gamma}$, where
$\overleftarrow{\gamma}$  denotes the cyclic structure with the
opposite orientation to $\gamma$.  Throughout this chapter we will study
$\tilde{\cal P}_S$, but the full vector space ${\cal P}_S$ will be studied
in chapter 3.
\end{defn}

\subsubsection{Shuffles of polygons}
Let $T_1,T_2$ denote two subsets of $Z=\{z_1,\ldots,z_n\}$
satisfying:
\begin{eqnarray}\label{Tprodconds1}
T_1\cup T_2 & =& Z \\
|T_1\cap T_2| & = & 3\ \nonumber.
\end{eqnarray}
Let $E=\{z_{i_1},z_{i_2},z_{i_3}\}$ denote the set of three points common to
$T_1$ and $T_2$.  

%
\begin{defn}
Consider elements $\gamma_1$ and $\gamma_2$ in $\tilde{\cal P}_S$ coming from
a choice of cyclic structure on $T_1$ and $T_2$ respectively.  For every such 
pair, define the {\it shuffle relative to the set $E$ of three points of
intersection}, $\gamma_1\sha_{\! E}\gamma_2$ by taking the unique liftings of 
$\gamma_1$ and $\gamma_2$ to elements $\bar\gamma_1$  and $\bar\gamma_2$ of ${\cal P}_S$ such 
that the cyclic order on $E$ obtained by restricting the cyclic order $\bar\gamma_1$ on $T_1$ 
(resp. $\bar\gamma_2$ on $T_2$) is equal to the standard cyclic order on $E$, and setting
\begin{equation}
\label{Gf}
\gamma_1\sha_{\! E}\gamma_2=\sum_{{{\bar\gamma\in{\cal P}_S}\atop{\bar\gamma|_{T_1}=\bar\gamma_1,
\bar\gamma|_{T_2}=\bar\gamma_2}}} \gamma,
\end{equation}
where $\gamma$ denotes the image in $\tilde{\cal P}_S$ of $\bar\gamma\in {\cal P}_S$.
\end{defn}

\vspace{.2cm}
We can write the shuffle with respect to three points using the following simple formula
(compare with (2.10)).  If $\{z_1,\ldots,z_n\}=\{0,1,\infty,t_1,\ldots,t_\ell\}$ with
$E=\{0,1,\infty\}$, we write $\gamma_1=(0, A_{1,2}, 1, A_{2,3},
\infty, A_{3,1})$ where $T_1$ is the disjoint union of
$A_{1,2}, A_{2,3}, A_{3,1}$ and $0,1,\infty$, and
$\gamma_2=(0, B_{1,2}, 1, B_{2,3}, \infty, B_{3,1})$, where $T_2$ is
the disjoint union of $B_{1,2}, B_{2,3}, B_{3,1}$ and $0,1,\infty$.
Then $\gamma_1\sha_{\! E} \gamma_2$ is the sum of polygons in $\tilde{\cal P}_S$ given by
$$\gamma= (0, A_{1,2}\sha B_{1,2}, 1, A_{2,3}\sha B_{2,3}, \infty, A_{3,1}\sha B_{3,1})\ .$$

\begin{example} \label{exsh1}  Let $T_1=\{0,1,\infty,t_1,t_3\}$ and
$T_2=\{0,1,\infty,t_2\}$.  Let $\gamma_1$ and $\gamma_2$ denote the elements of
$\tilde {\cal P}_S$ given by cyclic orders 
$(0,t_1,1,t_3,\infty)$ and $(0,\infty,t_2,1)$.  Then we take the liftings
$\bar\gamma_1=(0,t_1,1,t_3,\infty)$, $\bar \gamma_2=(-1)^4(0,1,t_2,\infty)$, 
and we find that 
$$\gamma_1\sha\gamma_2=(0,t_1,1,t_2,t_3,\infty)+(0,t_1,1,t_3,t_2,\infty)\in \tilde{\cal P}_S.$$
We will often write, for
example, $(0,t_1,1,t_2\sha t_3,\infty,t_4)$ for the right-hand side.
\end{example}

\vspace{.2cm}
\subsubsection{Multiplying pairs of polygons: the modular shuffle
  relation}
In this section, we consider elements of $\tilde{\cal P}_S \otimes 
\tilde{\cal P}_S$.  We use the notation $(\gamma,\eta)$ for 
$\gamma\otimes \eta$ where $\gamma,\eta\in \tilde{\cal P}_S$.  When $\gamma$
and $\eta$ are polygons (as opposed to linear combinations),
we can associate a geometric meaning to a pair
of polygons as follows. The left-hand polygon $\gamma$, which we
will write using round parentheses, for example
$(0,t_1,\ldots,t_\ell,1,\infty)$, is associated to the real cell
$X_\gamma$ of the moduli space $\Mod_{0,n}$ associated to the cyclic
structure. The right-hand polygon $\eta$, which we will write using
square parentheses, for example $[0,t_1,\ldots,t_\ell,1,\infty]$, is
associated to the cell-form $\omega_\eta$ associated to the cyclic
structure.  The pair of polygons will be associated to the (possibly
divergent) integral $\int_{X_\gamma}\omega_\eta$.  This geometric interpretation
extends in the obvious way to all pairs of elements $(\gamma,\eta)$.
In the following section we will investigate in detail the map from pairs of 
polygons to integrals.

\begin{defn}  Given sets $T_1, T_2$ as in (\ref{Tprodconds1}), the
{\it modular shuffle product} on the vector space $\tilde{\cal P}_S\otimes
\tilde{\cal P}_S$ is defined by
\begin{equation}\label{multpairs}
(\gamma_1,\eta_1)\sha (\gamma_2,\eta_2)=(\gamma_1\sha \gamma_2,
\eta_1\sha\eta_2),
\end{equation}
for pairs of polygons $(\gamma_1,\eta_1)\sha
(\gamma_2,\eta_2)$, where $\gamma_i$ and $\eta_i$ are cyclic structures
on $T_i$ for $i=1,2$.

\end{defn}

\begin{example} The following product of two polygon pairs is given by
\begin{align*}
\bigl(
(0,t_1,1,\infty, t_4), [0,\infty,t_1, t_4,1] \bigr) & \bigl( (
0,t_2,1,t_3,\infty), [0,t_3,t_2,\infty,1] \bigr)\\
& = - \bigl( (0,t_1\sha t_2, 1, t_3, \infty,t_4), [0,t_3,t_2,\infty,
t_1,t_4,1] \bigr).
\end{align*}
\end{example}

\vspace{.3cm} Let us now explain the geometric meaning of the modular shuffle
product (\ref{multpairs}), in terms of integrals of forms on moduli space.
Recall that a \emph{product map} between moduli spaces was defined
in\cite{Br2} as  follows. Let $T_1,T_2$ denote two subsets of
$Z=\{z_1,\ldots,z_n\}$ as in (\ref{Tprodconds1}), 
Then we can consider the product of forgetful maps:
\begin{equation}
f=f_{T_1}\times f_{T_2} : \Mod_{0,n} \To \Mod_{0,T_1} \times
\Mod_{0,T_2}\ .
\end{equation}
The map $f$ is a birational embedding
because $$\dim \Mod_{0,S} = |S|-3= |T_1|-3+|T_2|-3= \dim
\Mod_{0,T_1}\times \Mod_{0,T_2}\ .$$

If $f$ is a product map as above and
$z_i,z_j,z_k$ are the three common points of $T_1$ and $T_2$,
use an element $\alpha\in \PSL_2$ to map $z_i$ to $0$, $z_j$ to $1$
and $z_k$ to $\infty$.  Let $t_1,\ldots,t_\ell$ denote the images of
$z_1,\ldots,z_n$ (excluding $z_i,z_j,z_k$) under $\alpha$.
Given the indices $i$, $j$ and $k$, the product map is then determined by
specifying a partition of $\{t_1,\ldots,t_\ell\}$ into $S_1$ and $S_2$.
We use the notation $T_i=\{0,1,\infty\}
\cup S_i$ for $i=1,2$.

The shuffle product formula (\ref{multpairs}) on pairs of polygons
is motivated by the formula for multiplying integrals given in the following
proposition.

\begin{prop}\label{prodmaprel} Let $S=\{1,\ldots,n\}$, and let $T_1$ and $T_2$ be subsets
of $S$ as in (\ref{Tprodconds1}), of orders $r+3$ and $s+3$ respectively.
Let $\omega_1$ (resp. $\omega_2$) be a cell-form on
$\Mod_{0,r}$  (resp. on $\Mod_{0,s}$), and let $\gamma_1$ and $\gamma_2$
denote cyclic orderings on $T_1$ and $T_2$.  Then the product rule for
integrals is given by the following formula, called the modular shuffle
relation:
\begin{equation}\label{temp}\int_{X_{\gamma_1}} \omega_1\int_{X_{\gamma_2}}
\omega_2=\int_{X_{\gamma_1\sha\gamma_2}} \omega_1\sha \omega_2,
\end{equation}
where $\omega_1\sha\omega_2$ converges on the cell $X_\gamma$ for
each term $\gamma$ in $\gamma_1\sha\gamma_2$.

\end{prop}

\begin{proof} The subsets $T_1$ and $T_2$ correspond to a product  map
$$f:\Mod_{0,n}\rightarrow \Mod_{0,r}\times \Mod_{0,s}.$$
The pullback formula gives a multiplication law on the pair of integrals:
\begin{equation}\label{temp2}\int_{X_{\gamma_1}} \omega_1\int_{X_{\gamma_2}}
\omega_2= \int_{X_{\gamma_1}\times X_{\gamma_2}} \omega_1\wedge\omega_2=
\int_{f^{-1}(X_{\gamma_1}\times X_{\gamma_2})} f^*(\omega_1\wedge\omega_2).
\end{equation}
The preimage $f^{-1}(X_{\gamma_1}\times X_{\gamma_2})$ decomposes
into a disjoint union of cells of $\Mod_{0,n}$, which are precisely the cells
given by cyclic orders of $\gamma_1\sha \gamma_2$.
In other words,
$$f^{-1}(X_{\gamma_1}\times X_{\gamma_2})=\sum_{\gamma\in \gamma_1\sha\gamma_2}
X_{\gamma}\ ,$$ where the sum denotes a disjoint union. Now we can
assume without loss of generality that $T_1=\{0,1,\infty,
t_1,\ldots, t_k\}$, $T_2=\{0,1,\infty, t_{k+1},\ldots, t_\ell\}$ and
that $\delta_1,\delta_2$ are the cyclic structures on $T_1,T_2$
corresponding to $\omega_1,\omega_2$, respectively, where $\delta_1,
\delta_2$ restrict to the standard cyclic order  on $0,1,\infty$.
Then, in cell function notation,
$$f^*(\omega_{1}\wedge \omega_{2}) 
= \cfl \delta_1 \cfr \cfl \delta_2 \cfr \,dt_1\ldots dt_\ell= {\cfl
\delta_1 \sha_{\{0,1,\infty\}} \delta_2 \cfr \over \cfl
0,1,\infty\cfr} \,dt_1\ldots dt_\ell=\omega_1\sha \omega_2\ ,$$ by
proposition $\ref{shufprod}$. Since $\omega_1$ and $\omega_2$
 converge on the closed cells $\overline{X}_{\gamma_1}$ and
 $\overline{X}_{\gamma_2}$ respectively, $\omega_1\wedge \omega_2$
 has no poles on the contractible set $\overline{X}_{\gamma_1}\times \overline{X}_{\gamma_2},$
and therefore
 $\omega_1\sha\omega_2=f^*(\omega_1\wedge\omega_2)$ has no poles on the closure of
 $f^{-1}(X_{\gamma_1}\times X_{\gamma_2})$. But $\sum_{\gamma\in \gamma_1\sha\gamma_2}
X_{\gamma}$ is a cellular decomposition of
$f^{-1}(X_{\gamma_1}\times X_{\gamma_2})$, so, in particular,
$\omega_1 \sha \omega_2$ can have no poles along the closure of each
cell $ X_{\gamma},$ where $\gamma\in \gamma_1\sha\gamma_2$.
\end{proof}

\subsubsection{$\Sym(n)$ action on pairs of polygons}
The symmetric group $\Sym(n)$ acts on a pair of polygons by
permuting their labels in the obvious way, and this extends to the
vector space $\tilde{\cal P}_S\otimes \tilde{\cal P}_S$ by linearity. If
$\tau:\Mod_{0,n}\rightarrow \Mod_{0,n}$  is an element of $\Sym(n)$,
then the corresponding action on integrals is given by the pullback
formula:
\begin{equation}\label{permaction}
\int_{X_\gamma}\omega_\eta = \int_{\tau(X_\gamma)}
\tau^*(\omega_\eta) = \int_{X_{\tau(\gamma)}} \omega_{\tau(\eta)}\ .
\end{equation}
Suppose that $\tau$ belongs to the dihedral group which preserves
the dihedral structure underlying a cyclic structure $\gamma$. Let
$\epsilon=1$ if $\tau$ preserves $\gamma$, and $\epsilon=-1$ if
$\tau$ reverses its orientation. We have the following {\it dihedral
relation} between convergent integrals:
\begin{equation}\label{dihrel}
\int_{X_\gamma} \omega_\eta = (-1)^\epsilon\int_{X_\gamma}
\tau^*(\omega_\eta)=(-1)^\epsilon\int_{X_\gamma} \omega_{\tau(\eta)}.
\end{equation}

Both the formulas (\ref{permaction}) and (\ref{dihrel}) extend to linear 
combinations of integrals of cell-forms as long as the linear combination 
converges over the integration domain.  This convergence is not a consideration
when working with pairs of polygons rather than integrals.

\vspace{.2cm}
\begin{example}
The form corresponding to $\zeta(2,1)$ on $\Mod_{0,6}$ is
$$\frac{dt_1dt_2dt_3}{(1-t_1)(1-t_2)t_3} = [0,1,t_1,t_2,\infty,t_3]
+ [0,1,t_2,t_1,\infty,t_3],$$ which gives  $\zeta(2,1)$ after
integrating over the standard cell.  By applying the rotation
$(1,2,3,4,5,6)$, a dihedral rotation of the standard cell, to this
form, one obtains
\begin{align*}[t_1,\infty,t_2,t_3,0,1] +
  [t_1,\infty,t_3,t_2,0,1] & =
[0,1,t_1,\infty, t_2,t_3] +
[0,1,t_1,\infty,t_3,t_2]
\\ & ={{dt_1dt_2dt_3}\over{(1-t_1)t_2t_3}},\end{align*}
which gives $\zeta(3)$ after integrating over the standard cell.
Therefore, we have the following relation on linear combinations of
pairs of polygons:
\begin{equation}
\begin{split}
&\bigl( (0,t_1,t_2,t_3,1,\infty), [0,1,t_1,t_2,\infty,t_3] +
[0,1,t_2,t_1,\infty,t_3] \bigr)\\
&\qquad = \bigl( (0,t_1,t_2,t_3,1,\infty), [0,1,t_1,\infty, t_2,t_3] +
[0,1,t_1,\infty,t_3,t_2]\bigr)
\end{split}
\end{equation}
which on the level of integrals corresponds to
\begin{align*}
\zeta(2,1)=\int_{X_{3,\delta}}
\frac{dt_1dt_2dt_3}{t_3(1-t_2)(1-t_1)} &=
\int_{X_{3,\delta}} \frac{dt_1dt_2dt_3}{t_3t_2(1-t_1)}=\zeta(3). \\
\end{align*}
\end{example}
\begin{rem}This identity is an example of the well-known duality relation
between multiple zeta values given as follows.  Every tuple $(n_1,\ldots,n_r)$ 
of positive integers with $n_1>1$ is uniquely associated to a word 
$x^{n_1-1}y\cdots x^{n_r-1}y$ in non-commutative variables $x$, $y$.  Let
$(m_1,\ldots,m_s)$ be the tuple thus associated to the word 
$xy^{n_r-1}\cdots xy^{n_1-1}$.  The duality relation is
$$\zeta(n_1,\ldots,n_r)=\zeta(m_1,\ldots,m_s).$$
This relation follows from the dihedral relation above, using the reflection
permutation corresponding to the reflection of the polygon $(0,1,t_1,\ldots,
t_{n-3},\infty)$ over the symmetry axis through the side labeled $\infty$.
\end{rem}
\vspace{.3cm}
\subsubsection{Standard pairs and the product map relations}\label{prodmaps}
A standard pair of polygons is a pair $(\delta,\eta)$ where the
left-hand polygon is the standard cyclic structure.  Let
$S=\{1,\ldots,n\}$, and $T_1\cup T_2=S$ with $T_1\cap
T_2=\{0,1,\infty\}$  be as above, and let $\gamma_1$ and $\gamma_2$
be cyclic orders on $T_1$ and $T_2$. In the present section we show
how for each such $\gamma_1,\gamma_2$, we can modify the modular
shuffle relation to construct a multiplication law on standard
pairs.

\begin{defn}\label{prodmaprel2}
Let $\delta_1$ and $\delta_2$ denote the standard orders on $T_1$
and $T_2$.  Then there is a unique permutation $\tau_i$ mapping
$\delta_i$ to $\gamma_i$ such that $\tau_i(0)=0$, for $i=1,2$.  The
multiplication law, denoted by the symbol $\times$, and called the
{\it product map relation}, is defined by
\begin{equation}\label{pm}
\begin{split}
(\delta_1,\omega_1)\times(\delta_2,\omega_2)&=(\gamma_1,\tau_1(\omega_1))
\sha(\gamma_2,\tau_2(\omega_2))\\
&=(\gamma_1\sha\gamma_2,
\tau_1(\omega_1)\sha\tau_2(\omega_2))\\
&=\sum_{\gamma\in\gamma_1\sha\gamma_2}(\delta,\tau_\gamma^{-1}(
\tau_1(\omega_1)\sha\tau_2(\omega_2))),
\end{split}
\end{equation}
where for each $\gamma\in\gamma_1\sha\gamma_2$, $\tau_\gamma$ is
the unique permutation such that $\tau_\gamma(\delta)=\gamma$ and
$\tau_\gamma(0)=0$.
\end{defn}
\begin{example} Let $S=\{0,1,\infty,t_1,t_2,t_3,t_4\}$, $T_1=\{0,1,\infty,
t_1,t_4\}$ and $T_2=\{0,1,\infty,t_2,t_3\}$.  Let the cyclic orders
on $T_1$ and $T_2$ be given by $\gamma_1=(0,t_1,1,\infty,t_4)$ and
$\gamma_2=(0,t_2,1,t_3,\infty)$.  Applying the product map relation
to the pairs of polygons below yields
\begin{equation}
\begin{split}
\bigl((0,t_1,t_4,1,\infty),&[0,1,t_1,\infty,t_4]\bigr)\times
\bigl((0,t_2,t_3,1,\infty),[0,1,t_2,\infty,t_3]\bigr)\\
&=\bigl( (0,t_1,1,\infty, t_4), [0,\infty,t_1, t_4,1] \bigr) \sha\bigl( (
0,t_2,1,t_3,\infty), [0,t_3,t_2,\infty,1] \bigr)\\
& = - \bigl( (0,t_1,t_2, 1, t_3, \infty,t_4), [0,t_3,t_2,\infty,
t_1,t_4,1] \bigr)\\
 &\qquad\qquad\qquad - \bigl( (0,t_2,t_1, 1, t_3, \infty,t_4), [0,t_3,t_2,
\infty, t_1,t_4,1] \bigr)\\
&=\bigl((0,t_1,t_2,t_3,t_4,1,\infty),[0,t_3,\infty,t_1,1,t_2,t_4]+
[0,t_3,\infty,t_2,1,t_1,t_4].
\end{split}
\end{equation}
\end{example}

In terms of integrals, this corresponds to the relation
\begin{equation}
\begin{split}
\zeta(2)^2&=\int_{X_{5,\delta}} {{dt_1dt_4}\over{(1-t_1)t_4}}
\int_{X_{5,\delta}} {{dt_2dt_3}\over{(1-t_2)t_3}}\\
&=\int_{X_{7,\delta}} {{dt_1dt_2dt_3dt_4}\over{t_4(t_4-t_2)(1-t_2)(1-t_1) t_3}}
+{{dt_1dt_2dt_3dt_4}\over{t_4(t_4-t_1)(1-t_1)(1-t_2) t_3}}\\
\end{split}
\end{equation}
We will show in $\S\ref{calculations}$ that the last  two integrals
evaluate to  ${7\over 10}\zeta(2)^2$ and ${3\over 10}\zeta(2)^2$
respectively. \vspace{.3cm}
\subsection{The algebra of cell-zeta values} \label{cellalg}

\begin{defn}
Let ${\cal C}$ denote the $\Q$-subvector space of $\R$ generated by the
integrals $\int_{X_{n,\delta}} \omega$, where $X_{n,\delta}$ denotes
the standard cell of $\Mod_{0,n}$ for $n\ge 5$ and $\omega$ is a
holomorphic $\ell$-form on $\Mod_{0,n}$ with logarithmic
singularities at infinity (thus a linear combination of $01$
cell-forms) which converges on $X_{n,\delta}$.  We call these
numbers {\it cell-zeta values}.  The existence of product map
multiplication laws in proposition \ref{prodmaprel} imply that
${\cal C}$ is in fact a $\Q$-algebra.
\end{defn}

\begin{thm}\label{brownsthesis} The $\Q$-algebra ${\cal C}$ of cell-zeta values is isomorphic
to the $\Q$-algebra ${\cal Z}$ of multizeta values.
\end{thm}

\begin{proof} Multizeta values are real numbers which can all be
expressed as integrals $\int_{X_{n,\delta}} \omega$
where $\omega$ is an $\ell$-form of the form
\begin{equation}\label{Kont}
\omega=(-1)^d\prod_{i=1}^\ell {{d\underline{t}}\over{t_i-\epsilon_i}},
\end{equation}
where $\epsilon_1=0$, $\epsilon_i\in \{0,1\}$ for $2\le i\le
\ell-1$, $\epsilon_\ell=1$, and $d$ denotes the number of $i$ such
that $\epsilon_i=1$. Since each such form converges on
$X_{n,\delta}$, the multizeta algebra ${\cal Z}$ is a subalgebra of
${\cal C}$.  The converse is a consequence of the following theorem
due to F. Brown \cite{Br2}.
\begin{thm} If $\omega$ is a holomorphic $\ell$-form on $\Mod_{0,n}$ with
logarithmic singularities at infinity and convergent on
$X_{n,\delta}$, then $\int_{X_{n,\delta}} \omega$ is $\Q$-linear
combination of multizeta values.
\end{thm}
Thus, ${\cal C}$ is also a subalgebra of ${\cal Z}$, proving the equality.
\end{proof}

The structure of the formal multizeta algebra, generated by symbols (formally
representing integrals of the form (\ref{Kont})) subject to relations such as
shuffle and stuffle relations, has been much studied.  The present article 
provides a different approach to the study of this algebra, by turning instead 
to the study of a formal version of ${\cal C}$.

\begin{defn}\label{formalcell} Let $|S|\geq 5$.
The {\it formal algebra of cell-zeta values} ${\cal FC}$ is defined as
follows.  Let ${\cal A}$ be the vector space of formal linear combinations
of standard pairs of polygons in $\tilde{\cal P}_S\otimes \tilde{\cal P}_S$
$$\sum_i a_i(\delta,\omega_i)$$
such that the associated $\ell$-form $\sum_i a_i\omega_i$ converges
on the standard cell $X_{n,\delta}$.  Let ${\cal FC}$ denote the
quotient of ${\cal A}$ by the following families of relations.
\end{defn}

\begin{defn}\label{threerels} The three families of relations defining
${\cal FC}$ are as follows:
\begin{itemize}
\item{\it Product map relations.\ \ }{These relations were defined in section
\ref{prodmapsection}. For every choice of subsets $T_1, T_2$ of
$S=\{1,\ldots,n\}$ such that  $T_1\cup T_2=S$ and $|T_1\cap T_2|=3$,
and every choice of cyclic orders $\gamma_1, \gamma_2$ on $T_1,
T_2$, formula (\ref{pm}) gives a multiplication law expressing the
product of any two standard pairs of polygons of sizes $|T_1|$ and
$|T_2|$ as a linear combination of standard pairs of polygons of
size $n$.}
\item{\it Dihedral relations.\ \ }{For $\sigma$ in the dihedral group
associated to $\delta$, i.e. $\sigma(\delta)=\pm \delta$, there is a
dihedral relation  $(\delta,\omega)=(\sigma(\delta),
\sigma(\omega))$.}
\item{\it Shuffles with respect to one element.\ \ }{The linear combinations
of pairs of polygons $(\delta,(A,e)\sha_e (B,e))$ where $A$ and $B$ are
disjoint of length $n-1$ are zero, as in (\ref{cellfunction1shuffsare0}).}
\end{itemize}

\end{defn}

With the goal of approaching the combinatorial conjectures given in the
introduction, the purpose of the next chapters is to give an explicit
combinatorial description of a set of generators for ${\cal FC}$.  We do
this in two steps.  First we define the notion of a linear
combination of polygons convergent with respect to a chord of the standard
polygon $\delta$, and thence, the notion of a linear combination of
polygon convergent with respect to the standard polygon.  We
exhibit an explicit basis, the basis of {\it Lyndon insertion words and
shuffles} for the subspace of such linear combinations.  In the subsequent
chapter, we deduce from this a set of generators for the formal cell-zeta
value algebra ${\cal FC}$ and also, as a corollary, a basis for the subspace
of the cohomology space $H^\ell(\Mod_{0,n})$ consisting of classes of 
forms converging on the standard cell.

\vspace{.5cm}
\begin{rem}\label{laterremark} One of the most intriguing and important
questions concerning ${\cal FC}$ is the conjectural isomorphism with
the algebra of formal multizeta values ${\cal FZ}$ mentioned earlier
in conjecture \ref{wewish}.  In fact, there is a very natural
``candidate map'' from the generators of ${\cal FZ}$ to elements of
${\cal FC}$, coming from simply mapping the differential forms
in (\ref{itint}) to the corresponding form in the convergent cohomology
group $H^\ell(\Mod_{0,n}^\delta)$ (an explicit expression in terms of
the basis is given in formula (\ref{Kontstoins}) below).  
However, in order to yield an
algebra morphism, this map would have to respect the regularized
double shuffle relations on the multizeta values.  The shuffle relation
is easy to obtain on the images, using the shuffle product maps
corresponding to the partition of $(0,t_1,\ldots,t_\ell,1,\infty)$ into
$(0,t_1,\ldots,t_m,1,\infty)$ and $(0,t_{m+1},\ldots,t_\ell,1,\infty)$
for $2\le m\le \ell-2$ (cf.  \cite{Br2}).  Likewise, one could hope that
the stuffle relations would follow from the so-called stuffle product
maps defined in \cite{Br2}.  These maps can be expressed very simply
in terms of the cubical coordinates $x_1,\ldots,x_\ell$ defined by
$t_1=x_1\cdots x_{\ell}$, $t_2=x_2\cdots x_{\ell},\ldots,t_\ell=x_\ell$,
as
$$(0,x_1,\ldots,x_{\ell},1,\infty)\mapsto (0,x_1,\ldots,x_m,1,\infty)\times
(0,x_{m+1},\ldots,x_{\ell},1,\infty)$$
(it is easy to see that this is indeed a product map \cite{Br2}).  However,
computing the product of two multizeta values as a sum using this product
map yields a sum of cell-zeta values which is not obviously equal to
a sum of multiple zeta values (let alone the desired stuffle sum).

By a method due to P. Cartier, the stuffle relations on multizeta values
written as integrals of the differential forms $\omega$ in (\ref{itint}) 
written in cubical coordinates can be proved using variable
changes of the form
\begin{equation}\label{toomuch}
\int_{[0,1]^\ell} \omega = \int_{[0,1]^\ell} \sigma^*(\omega)
\end{equation}
for $\sigma$ any permutation of the $\ell$ coordinates $x_1,\ldots,x_{\ell}$.
We could choose to forcibly add the relations (\ref{toomuch}), for all
forms $\omega$ such that both $\omega$ and $\sigma^*(\omega)$ are
defined on $\Mod_{0,n}$ and convergent on the standard cell.  This
would ensure the validity of the stuffle relations on multiple zeta values
inside ${\cal FC}$.  However,
we have abstained from doing so in the hopes that some possibly weaker
conditions may be deduced from our relations and imply the stuffle, hence
giving a morphism ${\cal FZ}\rightarrow {\cal FC}$ with the definition of
${\cal FC}$ above.  This
certainly occurs experimentally up to $n=9$.  The paper \cite{Sou} by
I. Soud\`eres takes up this question in the context of motivic multiple
zeta values.

\end{rem}
\begin{rem} By analogy with the situation for mixed Tate motives and
formal multizeta values, we expect that the formal cell-zeta value
algebra will be a Hopf algebra.  However, we have not yet determined
an explicit coproduct.
\end{rem}

\vspace{.5cm}
\section{Polygons and convergence}
\vskip .5cm
The present chapter is devoted to redefining certain familiar geometric notions
from the moduli space situation: differential forms, divisors, convergence
of forms on cells, divergence of forms along divisors, residues, etc., 
in the completely combinatorial setting of polygons.

In this setting, the twin notions of cells and cell-forms are simultaneously
replaced by the single notion of a polygon, as explained in the previous 
chapters.  Boundary divisors then correspond to chords of polygons, and
the issues of divergence become entirely symmetric, with a chord of one 
polygon being ``a bad chord'' for another if the latter corresponds to a 
form which diverges along the divisor represented by the bad chord. 
This language makes it much easier to discuss residue calculations, 
convergence of linear combinations of polygons along bad chords, 
and most importantly, convergence of linear combinations of polygons with 
respect to the standard polygon $\delta$.  In the main result of this chapter,
we exhibit an explicit basis for the space of linear combinations of polygons 
convergent with respect to the standard polygon, consisting
of linear combinations called {\it Lyndon insertion words and 
Lyndon insertion shuffles}.  This
result will be key in the following chapter to determining an explicit
basis for the space of holomorphic differential $\ell$-forms on $\Mod_{0,n}$ 
with logarithmic singularities at the boundary, that converge on the
standard cell $\delta$.  The integrals of these basis elements, baptized 
{\it cell-zeta values}, form the basic generating set of our algebra of 
cell-zeta values, and it is the polygon construction given here that allows
us to define a set of formal cell-zeta values generating the corresponding,
combinatorially defined, formal cell-zeta algebra.
\vspace{.3cm}
\subsection{Bad chords and polygon convergence}
For any finite set $R$ of cardinality $n$, let ${\cal P}_R$ denote the 
$\Q$ vector space of linear combinations of {\it polygons on $R$}, i.e.
cyclic structures on $R$, identified with planar polygons with edges indexed 
by $R$, as in definition \ref{polspace} from section \ref{prodmapsection}.

Let ${\cal V}$ denote the free polynomial shuffle algebra on the alphabet
of positive integers, and let $V$ be the quotient of ${\cal V}$ by the
relations $w=0$ if $w$ is a word in which any letter appears more than once
(these relations imply that $w\sha w'=0$ if $w$ and $w'$ are not disjoint).
A basis for ${\cal V}$ is usually taken to be the set of all words $w$, 
but a theorem of Radford (\cite{Rad} or \cite{Re}, Theorem 6.1 (i)), 
gives an alternative basis for ${\cal V}$ which we use here.

\begin{defn}
Put the lexicographic ordering on the set of all words in a given ordered
alphabet ${\cal A}$.  A {\it Lyndon word} $w$ in the alphabet is a word having 
the following property: for every way of cutting the word $w$ into two 
non-trivial pieces $w_1$ and $w_2$ (so $w$ is the concatenation $w_1w_2$), 
the word $w_2$ is greater than $w$ itself for the lexicographical order.
The {\it Lyndon basis} for the vector space generated by words in ${\cal A}$
is given by Lyndon words and shuffles of Lyndon words.  
\end{defn}

Consider the image of the Lyndon basis of ${\cal V}$ under the quotient map
${\cal V}\rightarrow V$.  The elements of this basis which do not map to zero 
remain linearly independent in $V$, whose basis thus consists of Lyndon
words with distinct letters -- such a word is Lyndon if and only if the
smallest character appears on the left -- and shuffles of disjoint
Lyndon words with distinct letters.  Throughout this chapter, we work
in $V$, so that when we refer to a `word',
we automatically mean a word with distinct letters, and shuffles of such
words are zero unless the words are disjoint.  Let $V_S$ be the
subspace of $V$ spanned by the $n!$ words of length $n$ with distinct letters
in the characters of $S=\{1,\ldots,n\}$.  Then the Lyndon basis for $V_S$
is given by the $(n-1)!$ Lyndon words of degree $n$ and the
$(n-1)\cdot (n-1)!$ shuffles of disjoint Lyndon words
the union of whose letters is equal to $S$.

Recall from definition \ref{polspace} that the vector space ${\cal P}_S$ is 
generated by cyclic structures on $\{1,\ldots,n\}$, identified with
planar $n$-polygons with edges indexed
by $S$.  If we consider $(n+1)$-polygons with edges indexed by $S\cup \{d\}$
for some new letter $d\notin S$, we have a natural isomorphism
\begin{equation}\label{basiciso}
V_S \buildrel\sim\over\rightarrow {\cal P}_{S\cup \{d\}}
\end{equation}
given by writing each cyclic structure on $S\cup \{d\}$ as a word
on the letters of $S$ followed by the letter $d$.

\begin{defn}\label{IS}
Let $I_S\subset {\cal P}_{S\cup\{d\}}$ be the subspace linearly
generated by shuffles of polygons $(A\sha B,d)$, where $A\cup B=S$, 
$A\cap B=\emptyset$ and $A,B\ne \emptyset$.  Here, a shuffle
of polygons simply refers to the linear combination of polygons
indexed by the words in the shuffle sum $(A\sha B,d)$. 
\end{defn}

Then under
the isomorphism (\ref{basiciso}), $I_S$ is identified with the subspace of $V_S$
generated by the part of the Lyndon basis consisting of shuffles.  By a
slight abuse of notation, we use the same notation $I_S$ for the
corresponding subspaces of ${\cal P}_{S\cup\{d\}}$ and of $V_S$.

\vspace{.2cm}
\begin{defn}\label{defchords} Let $D=S_1\cup S_2$ denote a 
stable partition of $S$ (partition
into two disjoint subsets of order $\ge 2$).  Let $\gamma$ be a polygon on
$S$.  We say that the partition $D$ corresponds to a {\it chord of $\gamma$} if
the polygon $\gamma$ admits a chord which cuts $\gamma$ into two pieces
indexed by $S_1$ and $S_2$.  The sets $S_1$, $S_2$ are called {\it blocks}
associated to the chord $D$.  Thus, a chord divides $\gamma$ into two blocks, 
and the set of chords $\chi(\gamma)$ indexes the set of stable partitions which 
are {\it compatible} with $\gamma$ in the sense that the subsets 
$S_1$ and $S_2$ of the partition are blocks of $\gamma$.
\end{defn}

\vspace{.2cm}
\begin{defn} \label{defnconvpoly}
Let $\gamma, \eta$ denote two polygons on $S$. We say
that $\eta$ is \emph{convergent relative to} $\gamma$ if there are no stable
partitions of $S$ compatible with both $\gamma$ and $\eta$:
\begin{equation}\label{convcond}
\chi(\gamma)\cap \chi(\eta) = \emptyset\ .
\end{equation}
In other words, there exists no block of $\gamma$ having the same underlying
set as a block of $\eta$.  If $\eta$ is a polygon on $S$,
then a block of $\eta$ is said to be a {\it consecutive block} if
its underlying set corresponds to a block of the polygon with the
standard cyclic order $\delta$.  The polygon $\eta$ is said to be
{\it convergent} if it has no consecutive blocks at all, i.e., if it
is convergent relative to $\delta$.  A polygon $\eta\in
{\cal P}_{S\cup \{d\}}$ is said to be convergent if it has no
chords partitioning $S\cup \{d\}$ into disjoint subsets $S_1\cup S_2$
such that $S_1$ is a consecutive subset of $S=\{1,\ldots,n\}$.
\end{defn}

\vspace{.2cm}
\begin{defn} \label{defnconvword}  We now adapt the definition of convergence
for polygons in ${\cal P}_{S\cup\{d\}}$ to the corresponding words in $V_S$.
A {\it convergent word} in the alphabet $S$ is a word having no
subword which forms a consecutive block.  In other words, if
$w=a_{i_1}a_{i_2}\cdots a_{i_r}$, then $w$ is convergent if it
has no subword $a_{i_j}a_{i_{j+1}}\cdots a_{i_k}$ such that
the underlying set $\{a_{i_j},a_{i_{j+1}},\ldots,a_{i_k}\}=
\{i,i+1,\ldots,i+r\}\subset \{1,\ldots,n\}$.  A convergent word
is in fact the image in $V_S$ of a convergent polygon in
${\cal P}_{S\cup \{d\}}$ under the isomorphism (\ref{basiciso}).
\end{defn}

\begin{example}
When $1\leq n\leq 4$ there are no convergent polygons in ${\cal P}_S$. For
$n=5$, there is only one convergent polygon up to sign, given by
$\gamma=(13524)$.  The other convergent cyclic structure $(14253)$ is just
the cyclic structure $(13524)$ written backwards.
When $n=6$, there are three convergent polygons up to sign:
$$(135264) \ , \quad (152463) \ , \quad (142635) \ .$$
There are 23 convergent polygons for $n=7$. Note that when $n=8$, the
dihedral structure $\eta=(24136857)$ is not convergent even though
no neighbouring numbers are adjacent, because
$\{1,2,3,4\}$ forms a consecutive block for both $\eta$ and $\delta$.
\end{example}

\begin{rem}
The enumeration of permutations satisfying the single condition
that no two adjacent elements in $\gamma$ should be consecutive (the
case $k=2$) is known as the dinner table problem and is a classic
problem in enumerative combinatorics. The more general problem of
convergent words (arbitrary $k$)
seems not to have been studied previously. The
problems coincide for $n\leq 7$, but the counterexample for $n=8$
above shows that the problems are not equivalent for $n\geq 8$.
\end{rem}

\vskip .5cm
\subsection{Residues of polygons along chords}  In this section, we give
a combinatorial definition on polygons generalizing the notion of the
residue of a differential form at a boundary divisor along which it
diverges.

\vspace{.3cm}
\begin{defn}\label{residues} ({\it Polygon residues})
For every stable partition $D$ of $S$ given by $S=S_1\cup S_2$,
we define a residue map on polygons
$$\Res^p_D:{\cal P}_S \To {\cal P}_{S_1\cup\{d\}} \otimes_\Q {\cal
  P}_{S_2\cup\{d\}}$$
as follows.  Let $\eta$ be a polygon in ${\cal P}_S$.  If the partition
$D$ corresponds to a chord of $\eta$, then it cuts $\eta$ into two subpolygons
$\eta_i$ ($i=1,2$) whose edges are indexed by the set $S_i$ and an edge
labelled $d$ corresponding to the chord $D$.  We set
\begin{equation}
\Res^p_D(\eta)=
\begin{cases}
\eta_1\otimes \eta_2&\hbox{if $D$ is a chord of
  $\eta$}\\
0&\hbox{if $D$ is not a chord of $\eta$}.
\end{cases}
\end{equation}

More generally, we can define the residue for several disjoint chords
simultaneously.
Let $S=S_1\cup \cdots \cup S_{r+1}$ be a partition of
$S$ into $r+1$ disjoint subsets with $r\ge 2$.  For $1\le i\le r$, let
$D_i$ be the partition of $S$ into the two subsets
$(S_1\cup \cdots S_i)\cup (S_{i+1}\cup
\cdots \cup S_{r+1})$.  For any polygon
$\eta\in {\cal P}_S$, we say that $\eta$ admits the chords $D_1,\ldots,D_r$
if there exist $r$ chords of $\eta$, disjoint except possibly for endpoints,
partitioning the edges of $\eta$
into the sets $S_1,\ldots,S_{r+1}$.  If $\eta$ admits the chords
$D_1,\ldots,D_r$, then these chords cut
$\eta$ into $r+1$ subpolygons $\eta_1,\ldots,\eta_{r+1}$.  Let
$T_i$ denote the set indexing the edges of $\eta_i$, so that each
$T_i$ is a union of $S_i$ and elements of the set
$\{d_1,\ldots,d_r\}$ of indices of the chords.  The {\it composed residue map}
$$\Res^p_{D_1,\ldots,D_r}:{\cal P}_S\rightarrow
{\cal P}_{T_1}\otimes \cdots \otimes {\cal P}_{T_r}$$
is defined as follows:
\begin{equation}
\Res^p_{D_1,\ldots,D_r}(\eta)=
\begin{cases}
\eta_1\otimes \cdots \otimes\eta_{r+1}&\mbox{if }\eta\mbox{ admits
}D_1,\ldots,D_r\mbox{ as  disjoint chords}\\
0&\mbox{if }\eta\mbox{ does not admit }D_1,\ldots,D_r
\end{cases}
\end{equation}
\end{defn}
\begin{ex}
In this example,  $n=12$ and the partition of $S$ given by $D_1$, $D_2$, $D_3$
and $D_4$ is $S_1=\{1,2,3\}$,
$S_2=\{4,10,11,12\}$, $S_3=\{5,9\}$, $S_4=\{6\}$, $S_5=\{7,8\}$.
\vskip .3cm
\epsfxsize=12cm
\centerline{\epsfbox{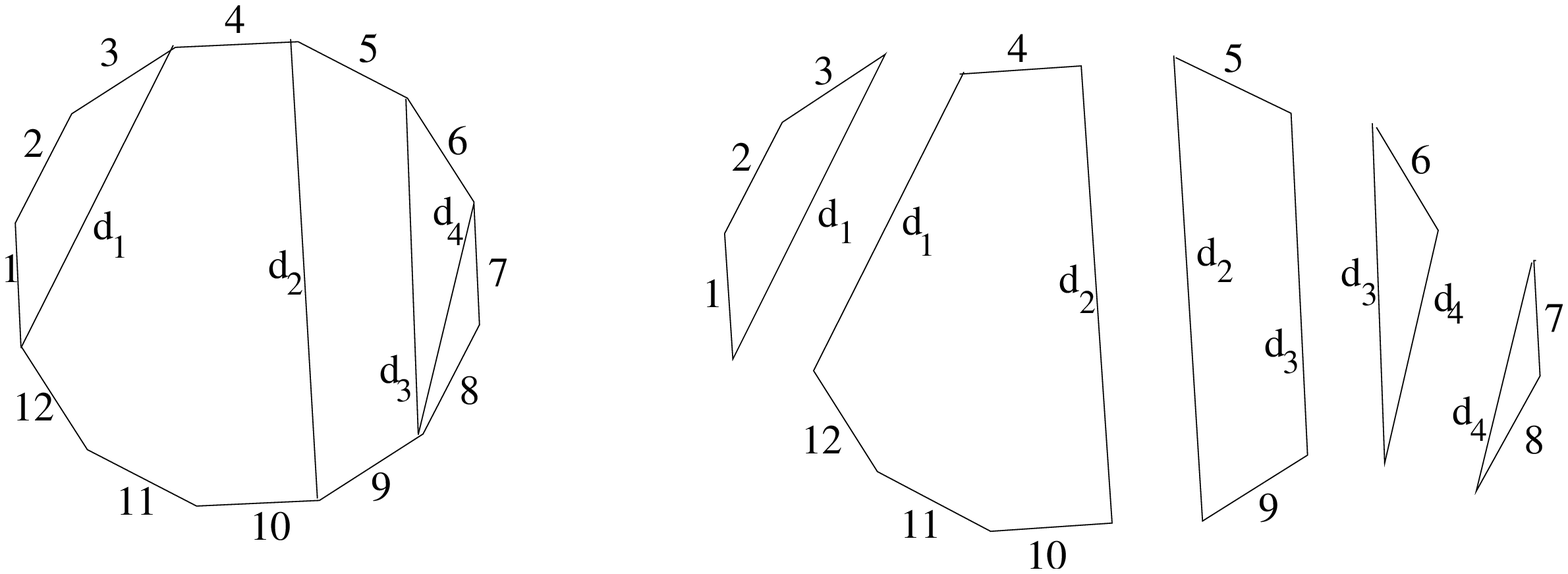}}
\vskip .3cm
We have $T_1=S_1\cup\{d_1\}$, $T_2=S_2\cup \{d_1,d_2\}$,
$T_3=S_3\cup \{d_2,d_3\}$, $T_4=S_4\cup \{d_3,d_4\}$, $T_5=S_5\cup \{d_4\}$.
The composed residue map $\Res^p_{D_1,D_2,D_3,D_4}$ maps the standard
polygon $\delta=(1,2,3,4,5,6,7,8,9,10,11,12)$ to the tensor product of the
five subpolygons shown in the figure.
\end{ex}

The definition of the residue allows us to extend the definition of
convergence of a polygon to linear combinations of polygons.
\vspace{.2cm}
\begin{defn} ({\it Polygon divergence along the standard polygon: bad chords}) Let $E$ be a partition of $S\cup \{d\}$ into two subsets,
one of which is a consecutive subset $T=\{i,i+1,\ldots,i+j\}$ of 
$S$ for the standard order, and let $\eta$ be a polygon.  We say that
$E$ is a {\it bad chord} for $\eta$, or eqiuvalently $\eta$ is a 
{\it bad polygon} for $E$, if $E\in \chi(\eta)$
(this expresses the idea that the cell-form corresponding to $\eta$ 
diverges along the boundary divisor, corresponding to $E$, of the
standard cell $\delta$).  If $\eta =\sum_i a_i\eta_i$, then we say that $E$
is a bad chord for $\eta$ if any $\eta_i$ is a bad polygon for $E$.
\end{defn}
\begin{defn}\label{convergence} ({\it Polygon convergence along the standard polygon}) 
The linear combination $\eta=\sum_i a_i\eta_i$ is said to {\it converge} 
along the chord $E$ of the standard polygon (or along the corresponding
consecutive subset $T$) if the residue satisfies
\begin{equation}\label{conv}
\Res^p_E(\eta)\in I_T\otimes {\cal P}_{S\setminus T\cup \{d\}\cup\{e\}},
\end{equation}
where $I_T$ is as in definition \ref{IS}.
A linear combination $\eta$ is {\it convergent} (along the standard
polygon) if it converges along all of its bad chords.
\end{defn}

\vspace{.2cm}
The goal of the following section is to define a set of particular
linear combinations of polygons, the Lyndon insertion words and 
Lyndon insertion shuffles,
which are convergent, and show that they are linearly independent.  In
the section after that, we will prove that this set forms a basis for the
convergent subspace of ${\cal P}_{S\cup \{d\}}$.
\vskip .5cm
\subsection{The Lyndon insertion subspace}\label{Lyndins}

\vspace{.3cm} 
\begin{defn}\label{WS}
Let a $1n$-word be a word of length $n$ in the distinct
letters of $S=\{1,\ldots,n\}$ in which the letter $1$ appears just to the
left of the letter $n$, and let $W_S\subset V_S\simeq {\cal P}_{S\cup\{d\}}$
denote the subspace generated by these words.   The space $W_S$ is of dimension
$(n-1)!$.  
\end{defn}

The following lemma will show that $V_S=W_S\oplus I_S$, where
$I_S$ is the subspace of shuffles of definition \ref{IS}.

\begin{lem}\label{usefullemma}
Fix two elements $a_1$ and $a_2$ of $S=\{1,\ldots,n\}$.  

\vspace{.2cm}
Let
\begin{equation*}
\tau=\sum_i c_i\eta_i\in V_S,
\end{equation*}
where the $\eta_i$ run over the words of length $n$ in $V_S$ such that
$a_1$ is the leftmost character of $\eta_i$ (resp. the $\eta_i$ run over the
words where $a_1$ appears just to the left of $a_2$ in $\eta_i$). Then 
$\tau \in I_S$ if and only if $c_i=0$ for all $i$. 
\end{lem}

\begin{proof} The assumption $\tau\in I_S$ means that we can write $\tau=
\sum_i c_i u_i\sha v_i$ for non-empty words $u_i$ and $v_i$. Considering  
this in the space ${\cal P}_{S\cup\{d\}}$
isomorphic to $V_S$, it is a sum of cyclic structures $\sum_i c_i(u_i,d)
\sha (v_i,d)$ shuffled with respect to the point $d$.  Choose any bijection 
$\rho:\{1,\ldots,n,d\}\rightarrow 
\{0,1,\infty,t_1,\ldots,t_{n-2}\}$ which maps $d$ to $0$ and $a_1$ to $1$
(resp. which maps $a_1$ to $0$ and $a_2$ to $1$).  Define a linear map from
${\cal P}_{S\cup \{d\}}$ to $H^{n-2}(\Mod_{0,n+1})$ by first renumbering
the indices $(1,\ldots,n,d)$ of each polygon $\eta\in {\cal P}_{S\cup \{d\}}$
as $(0,1,\infty,t_1,\ldots,t_{n-2})$ via $\rho$, then mapping the
renumbered polygon to the corresponding cell-form (same cyclic order).
By hypothesis,
$\tau=\sum_i c_i \eta_i$ maps to a sum $\omega_\tau=\sum_i c_i\omega_{\eta_i}$
of $01$ cell forms.  Since $\tau$ is a shuffle with respect to one point,
we know by (\ref{cellfunction1shuffsare0}) that $\omega_\tau=0$.
But the $01$ cell-forms $\omega_{\eta_i}$ are linearly
independent by theorem \ref{thm01cellsspan}.  Therefore each $c_i=0$.
\end{proof}

Recall that the shuffles of disjoint Lyndon words form a basis for $I_S$;
we call them {\it Lyndon shuffles}.  A {\it convergent Lyndon shuffle} is a
shuffle of convergent Lyndon words.

\begin{defn}\label{defngrr}
We will recursively define the set $\cal{L}_S$ of {\it Lyndon insertion
shuffles} in $I_S$.  If $S=\{1\}$, then ${\cal L}_S=\emptyset$.
If $S=\{1,2\}$ then ${\cal L}_S=
\{1\sha  2\}$. In general, if $D$ is any (lexicographically ordered)
alphabet on $m$ letters and $S=\{1,\ldots,m\}$, we define
${\cal L}_D$ to be the image of ${\cal L}_S$
under the order-preserving
bijection $S\rightarrow D$ corresponding to the ordering of $D$.

Assume now that $S=\{1,\ldots,n\}$ with $n>2$, and that we have
constructed all of the sets
${\cal L}_{\{1,\ldots, i\}}$ with $i<n$.  Let us construct
${\cal L}_S$.  The elements of these sets are constructed
by taking convergent Lyndon shuffles on a smaller alphabet, and
making ``insertions'' into every letter except for the
leftmost letter of each Lyndon word in the shuffle,
according to the following explicit
procedure.  Let $T=\{a_1,\ldots,a_k\}$ be an alphabet with $3\le k\le n$
letters, ordered by the lexicographical ordering $a_1<\cdots <a_k$,
and choose a convergent Lyndon shuffle $\gamma$ of length $k$ in the
letters of $T$.  Write $\gamma$ as a shuffle of $s>1$ convergent Lyndon words
in disjoint letters:
$$\gamma=(a_{i_1}\cdots a_{i_{k_1}})\sha  (a_{i_{k_1+1}}\cdots a_{i_{k_2}})
\sha \cdots\sha  (a_{i_{k_{s-1}+1}}\cdots a_{i_{k_s}})$$
where $1\le k_1<k_2<\cdots<k_s=k$.  Choose integers
$v_1,\ldots,v_k\ge 1$ such that $\sum_i v_i=n$ and such that for each of the
indices $l=i_1,i_{k_1+1},\ldots,i_{k_{s-1}+1}$ of the leftmost characters of the
$s$ convergent Lyndon words in $\gamma$, we have $v_l=1$.
For $1\le i\le k$, let $D_i$ denote an alphabet $\{b^i_1,\ldots,b^i_{v_i}\}$.
When $v_i=1$, insert $b^i_1$ into the place of the letter $a_i$ in
$\gamma$; when $v_i>1$, choose any element $V_i$ from ${\cal L}_{D_i}$,
and insert this $V_i$ into the place of the letter $a_i$.

The result is a sum of words in the alphabet
$\cup_{i=1}^k D_i$.  Note that this alphabet is of cardinal $n$ and equipped
with a natural lexicographical ordering given by the ordering
$D_1,\ldots,D_k$ and the orderings within each alphabet $D_i$.  We
can therefore renumber this alphabet as $1,\ldots,n$.
Since it is a sum of shuffles, the renumbered
element lies in $I_S$, and we call it a {\it Lyndon insertion shuffle} on $S$.
The original convergent Lyndon shuffle $\gamma$ on $T$ is called the framing;
together with the integers $v_i$, we call this the fixed structure
of the insertion shuffle.
We define ${\cal L}_S$ to be the set of all Lyndon insertion shuffles on $S$,
constructed by varying the choice of $3\le k\le n$, the convergent Lyndon 
shuffle $\gamma$ on $k$ letters, the numbers $v_1,\ldots,v_k$ and the
elements $V_i$ for each $v_i>1$ in every possible way.

In the special case where $k=n$, we have $v_i=1$ for $1\le i\le k$ and there 
are no non-trivial insertions.  The corresponding elements of ${\cal L}_S$ are
thus just convergent Lyndon shuffles.
\end{defn}

\begin{example}\label{examplegrr}
We have
$${\cal L}_{\{1,2\}}=\{1\sha  2\}$$
$${\cal L}_{\{1,2,3\}}=\{1\sha  2\sha  3, \ 2\sha  13\}$$
$${\cal L}_{\{1,2,3,4\}}=\{1\sha  2\sha  3\sha  4,\ 13\sha  2\sha  4,\
14\sha  2\sha  3,\ 24\sha  1\sha  3, $$
$$\qquad \qquad \qquad 3\sha  142,\ 13\sha  24,\ 1(3\sha  4)\sha  2\}$$
The last element of ${\cal L}_{\{1,2,3,4\}}$ is obtained by taking
$T=\{1,2,3\}$ and $\gamma=13\sha  2$.  We can only insert in the
place of the character 3 since 1 and 2 are leftmost letters of the
Lyndon words in $13\sha  2$.  As for what can be inserted in the
place of 3, the only possible choices are $k=1$, $v_1=2$,
$D_1=\{b_1,b_2\}$, and $V_1=b_1\sha  b_2$, the unique element of
${\cal L}_{D_1}$.  The natural ordering on the alphabet
$\{T\setminus 3\}\cup D_1$ is given by $(1,2,b_1,b_2)$ since
$b_1\sha  b_2$ is inserted in the place of 3, so we renumber $b_1$
as 3 and $b_2$ as 4, obtaining the new element $1(3\sha  4)\sha  2=
134\sha 2+143\sha 2=2134+1234+1324+1342+2143+1243+1423+1432$.

For $n=5$, ${\cal L}_{\{1,2,3,4,5\}}$ has 34 elements.  Of these, 25
are convergent Lyndon shuffles which we do not list.  The remaining
nine elements are obtained by insertions into the smaller convergent
Lyndon shuffles: they are given by
$$\begin{cases}
2\sha  1(4\sha  35),\ 2\sha  1(3\sha  4\sha  5)&\mbox{insertions into }2\sha  13\\
3\sha  1(4\sha  5)2,\ 4\sha  15(2\sha  3)&\mbox{insertions into }3\sha  142\\
13\sha  2(4\sha  5),\ 1(3\sha  4)\sha  25&\mbox{insertions into }13\sha  24\\
1(3\sha  4)\sha  2\sha  5&\mbox{insertion into }13\sha  2\sha  4\\
1(4\sha  5)\sha  2\sha  3&\mbox{insertion into }14\sha  2\sha  3\\
2(4\sha  5)\sha  1\sha  3&\mbox{insertion into }24\sha  1\sha  3.
\end{cases}$$

\end{example}

\begin{defn}\label{insw}
We now define a complementary set, the set ${\cal W}_S$ of
{\it Lyndon insertion words}.  Let a {\it special convergent word} $w\in V_S$
denote a convergent word of length $n$ in $S$ such that in the
lexicographical ordering $(1,\ldots,n,d)$, the polygon (cyclic structure)
$\eta=(w,d)$ satisfies $\chi(\delta)\cap \chi(\eta)=\emptyset$; in
other words, the polygon $\eta$ has no chords in common with the standard
polygon.  This condition is a little stronger than asking $w$
to be a convergent word (for instance, $13524$ is a convergent word but
not a special convergent word, since $13524d$ has a bad chord $\{2,3,4,5\}$).
The first elements of ${\cal W}_S$ are given by the special convergent
$1n$-words.  The remaining elements of ${\cal W}_S$ are the
{\it Lyndon insertion words} constructed as follows.  Take a special convergent
word $w'$ in a smaller alphabet $T=\{a_1,\ldots,a_k\}$ with $k<n$ such that
$a_1$ appears just to the left of $a_{k-1}$, and choose positive integers
$v_1,\ldots,v_k$ such that $v_1=v_k=1$ and $\sum_i v_i=n$.  As above, we let
$D_i=\{b^i_1,\ldots,b^i_{v_i}\}$ for $1\le i\le k$, and choose
an element $D_i$ of ${\cal L}_{D_i}$ for each $i$ such that $v_i>1$.  For $i$
such that $v_i=1$, insert $b^i_1$ in the place of $a_i$ in $w'$, and for $i$
such that $v_i>1$ insert $D_i$ in the place of $a_i$.  We obtain a sum of words
$w''$ in the letters $\cup D_i$.  This alphabet has a natural lexicographic
ordering $D_1,\ldots,D_k$ as above, so we can renumber its letters from $1$ to
$n$, which transforms $w''$ into a sum of words $w\in V_S$ called a
{\it Lyndon insertion word}.  Note that by construction, the result is still
a sum of $1n$-words.  The set ${\cal W}_S$ consists of the special convergent
words and the Lyndon insertion words.
\end{defn}

\begin{rem}It follows from lemma \ref{usefullemma} that the intersection
of the subspace $\langle {\cal W}_S\rangle$ in $V_S$ with the subspace
$I_S$ of shuffles is equal to zero.
\end{rem}

\begin{example} We have
$${\cal W}_{\{1,2\}}=\emptyset, \ \
{\cal W}_{\{1,2,3\}}=\emptyset,\ \
{\cal W}_{\{1,2,3,4\}}=\{3142\},$$
$${\cal W}_{\{1,2,3,4,5\}}=\{24153,31524,(3\sha 4)152,415(2\sha 3)\}$$
The last two elements of ${\cal W}_{\{1,2,3,4,5\}}$ are obtained by
taking $v_1=1,v_2=1,v_3=2,v_4=1$ and $v_1=1,v_2=2,v_3=1,v_4=1$ and
creating the corresponding Lyndon insertion word with respect to $3142$.
\end{example}

\vspace{.2cm}
\begin{thm}\label{linind}
The set ${\cal W}_S\cup {\cal L}_S$ of Lyndon insertion words and shuffles
is linearly independent.
\end{thm}
\begin{proof}
We will prove the result by induction on $n$.  Since ${\cal L}_S\subset
I_S$ and we saw by lemma \ref{usefullemma} that the space generated by
${\cal W}_S$ has zero intersection with $I_S$, we only have to show that
that both ${\cal W}_S$ and ${\cal L}_S$ are linearly independent sets.
We begin with ${\cal L}_S$.  Since
${\cal L}_{\{1,2\}}$ contains a single element, we may assume that $n>2$.

Let $W=A_1\sha \cdots\sha  A_r$ be a Lyndon shuffle, with $r>1$.
We define its {\it fixed
structure} as follows.  Replace every maximal consecutive block (not
contained in any larger consecutive block) in each $A_i$ by a single
letter.  Then $W$ becomes becomes a convergent Lyndon shuffle $W'$ in a
smaller alphabet $T'$ on $k$ letters, which is equipped with an inherited
lexicographical ordering.  If $T=\{1,\ldots,k\}$, then under the
order-respecting bijection $T'\rightarrow T$, $W'$ is mapped to a convergent
Lyndon shuffle $V$ in $T$, called the framing of $W$.
The fixed structure is given by the framing together with
the set of integers $\{v_i\mid 1\le i\le k\}$ defined by $v_i=1$ if that letter
in $T$ does not correspond to a maximal block, and $v_i$ is the length of the
maximal block if it does.  Thus we have $v_1+\cdots+v_k=n$.  We can extend
this definition to the fixed structure of a Lyndon insertion
shuffle, since by definition this is a linear combination of Lyndon
shuffles all having the same fixed structure, and we recover the
framing and fixed structure of the insertion shuffle given in the definition.

\begin{example} If $W$ is the Lyndon shuffle $1546\sha 237$, we
replace the consecutive blocks $23$ and $546$ by letters $b_1$ and $b_2$,
obtaining the convergent shuffle $W'=1b_2\sha b_17$ in the alphabet
$T'=\{1,b_1,b_2,7\}$; renumbering this as $1,2,3,4$ we obtain $V=13\sha 24
\in {\cal L}_{\{1,2,3,4\}}$.  The fixed structure is given by $13\sha 24$ and
integers $v_1=1,v_2=2,v_3=3, v_4=1$.

The Lyndon insertion shuffles $(1, (3 \sha  4 )) \sha  (2, 5 )$ and $(1,3) \sha
(2, (4 \sha  5))$ have the same framing
$13\sha  24$, but since
$(v_1,v_2,v_3,v_4)=(1,1,2,1)$ for the first one and $(1,1,1,2)$ for the
second, they do not have the same fixed structure.  The Lyndon insertion
shuffles $(1, (5) \sha  (3, 4, 6)) \sha  (2, 7) $ and $
(1, (3,5) \sha  (4,6)) \sha  (2,7)$ have the same associated framing
$13\sha  24$ and the same integers $(v_1,v_2,v_3,v_4)=(1,1,4,1)$.
so they have the same fixed structure.
\end{example}
\vskip .2cm
For any fixed structure, given by a convergent Lyndon shuffle $\gamma$
on an alphabet $T$ of length $k$ and associated integers $v_1,\ldots,v_k$
with $v_1+\cdots+v_k=n$, let
$L(\gamma, v_1, ..., v_k)$ be the subspace of
$V_S$ spanned by Lyndon shuffles with that fixed structure.  Since
Lyndon shuffles are linearly independent, we have
\begin{align*}
V_{S}=\bigoplus L(\gamma, v_1, ..., v_k)
\end{align*}
Now, as we saw above, a Lyndon insertion shuffle is a linear combination
of Lyndon shuffles all having the same fixed structure, so every
element of ${\cal W}_S\cup{\cal L}_S$ lies in exactly one subspace
$L(\gamma,v_1,\ldots,v_k)$.
Thus, to prove that the elements of ${\cal L}_S$ are
linearly independent, it is only necessary to prove the linear
independence of Lyndon insertion shuffles with the same fixed structure.
If all of the $v_i=1$, then the fixed structure is just a single
convergent Lyndon shuffle on $S$, and these are linearly independent.
So let $(\gamma,v_1,\ldots,v_k)$ be a fixed structure with not all of the
$v_i$ equal to $1$, and let $\omega=\sum_q c_q \omega_q$ be a linear
combination of Lyndon insertion shuffles of fixed structure
$\gamma,v_1,\ldots,v_k$.

Break up the tuple $(1,\ldots,n)$ into $k$ successive tuples
$$B_1=(1,\ldots,v_1),\ B_2=(v_1+1,\ldots,v_1+v_2),\ldots,
B_k=(v_1+\cdots+v_{k-1}+1,\ldots, n).$$

Let $i_1,\ldots,i_m$ be the indices such that $B_{i_1},\ldots,B_{i_m}$ are
the tuples of length greater than $1$.  These tuples correspond to the
insertions in the Lyndon insertion shuffles of type $(\gamma,v_1,\ldots,
v_k)$.  For $1\le j\le m$, let $T_j=\{B_{i_j}\}\cup \{d_j\}$.  This element
$d_j$ is the index of the chord $D_j$ corresponding to the consecutive subset
$B_{i_j}$, which is a chord of the standard polygon and also of every term of
$\omega$.  The chords $D_1,\ldots,D_r$ are disjoint and cut each term of
$\omega$ into $m+1$ subpolygons, $m$ of which are indexed by $T_j$, and the
last one of which is indexed by $T'=
S\setminus \{B_{i_1}\cup \cdots\cup B_{i_m}\}\cup \{d_1,\ldots,
d_m\}$.  Thus we can take the composed residue map
$$\Res^p_{D_1,\ldots,D_m}(\omega)\in {\cal P}_{T_1}\otimes
\cdots \otimes {\cal P}_{T_m}\otimes {\cal P}_{T'}.$$
Let us compute this residue.

The alphabet $T'$ is of length $k$ and has a natural ordering corresponding
to a bijection $\{1,\ldots,k\}\rightarrow T'$.  Let $\gamma'$ be the
image of $\gamma$ under this bijection, i.e. the framing.
Let $P^q_1,\ldots,P^q_m$ be the insertions corresponding to the $m$ tuples
$B_{i_1},\ldots,B_{i_m}$ in each term of $\omega=\sum_qc_q\omega_q$. Each
$P^q_j$ lies in ${\cal L}_{B_{i_j}}$.
The image of the composed residue map is then
\begin{equation}\label{compres}
\Res^p_{D_1,\ldots,D_m}(\omega)=\sum_q c_q (P^q_1,d_1)\otimes \cdots\otimes
(P^q_m,d_m)\otimes \gamma'.
\end{equation}
Now assume that $\omega=\sum_q c_q\omega_q=0$, and let us show that each
$c_q=0$. We have
$$\sum_q c_q (P^q_1,d_1)\otimes \cdots\otimes (P^q_m,d_m)\otimes \gamma'=0,$$
and since $\gamma'$ is fixed, we have
$$\sum_q c_q (P^q_1,d_1)\otimes \cdots\otimes (P^q_m,d_m)=0.$$
But for $1\le j\le m$, the $P^q_j$ lie in ${\cal L}_{B_{i_j}}$ and
thus, by the induction hypothesis, the distinct $P^q_j$ for fixed $j$
and varying $q$ are linearly independent.  Since $d_i$ is the largest
element in the lexicographic alphabet $T_i$, the sums $(P^q_j,d_j)$ are
also linearly independent for fixed $j$ and varying $q$, because if
$\sum_q e_q(P^q_j,d_j)=0$ then $\sum_q e_q P^q_j=0$ simply by erasing $d_j$.
The tensor products are therefore also linearly independent, so we must have
$c_q=0$ for all $q$.  This proves that ${\cal L}_S$ is a linearly independent
set.

We now prove that ${\cal W}_S$ is a linearly independent set.  For this,
we construct the framing and fixed structure of a Lyndon insertion word
of length $n$ in ${\cal W}_S$
just as above, by replacing consecutive blocks with single letters, obtaining a word in a smaller alphabet $T'$ and a set of integers corresponding to the
lengths of the consecutive blocks.  For instance, replacing the consecutive
block $(3\sha 4)$ in the Lyndon insertion word $(3\sha 4)152$ by the
letter $b_1$ gives a convergent word $b_1152$ in the alphabet $(1,2,b_1,5)$; 
renumbering this as $(1,2,3,4)$ gives the framing as $3124$ and the associated 
integers as $v_1=2,v_2=1,v_3=2,v_4=1$.  For every fixed structure of this type,
now given as a convergent word $\gamma$ of length $k<n$ together with integers
$v_1,\ldots,v_k$, we let $W(\gamma,v_1,\ldots,v_k)$ denote the subspace
of $V_S$ generated by Lyndon insertion words with the fixed structure 
$(\gamma,v_1,\ldots, v_k)$.  As above, the spaces $W(\gamma,v_1,\ldots,v_k)$
do not intersect, so ${\cal W}_S=\oplus W(\gamma,v_1,\ldots,v_k)$,  and we
have only to show that the set
of Lyndon insertion words with a given fixed structure is a linearly
independent set.  So assume that we have some linear combination
$\sum_q c_qw_q=0$, where the $w_q$ are all Lyndon insertion words of
given fixed structure $(\gamma,v_1,\ldots,v_k)$.  If $k=n$, then these
insertion words are just words, so they are linearly independent and
$c_q=0$ for all $q$.  So assume that at least one $v_i>1$.  We proceed
exactly as above.  Breaking up the tuple $(1,\ldots,n)$ into tuples
$B_1,\ldots,B_k$ as above, and letting $D_1,\ldots,D_m$, $T_j$ and $T'$
denote the same objects as before, we compute the composed residue of $\sum_q
c_qw_q$ and obtain (\ref{compres}).  Then because all of the insertions
$P^q_i$ lie in ${\cal L}_{B_{i_j}}$ and we know that these sets are
linearly independent, we find as above that $c_q=0$ for all $q$.
\end{proof}

\vspace{.3cm}
\subsection{Convergent linear combinations of polygons}

\begin{defn} Let $S=\{1,\ldots,n\}$.
Let $J_S$ be the subspace of ${\cal P}_{S\cup \{d\}}$ spanned
by the set ${\cal L}_S$ of Lyndon insertion shuffles, and let $K_S$ be the 
subspace of ${\cal P}_{S\cup \{d\}}$ spanned by the set ${\cal W}_S$
of Lyndon insertion words.
\end{defn}
\vskip .3cm
We prove the main convergence results in two separate theorems, concerning
the subspaces $I_S$ and $W_S$ of $V_S\simeq {\cal P}_{S\cup \{d\}}$
respectively (cf. definitions \ref{IS} and \ref{WS}).

\vspace{.2cm}
\begin{thm}\label{convJS}
An element $\omega\in I_S\subset {\cal P}_{S\cup \{d\}}$ is convergent if and
only if $\omega\in J_S$.
\end{thm}
\begin{proof} {\it Step 1.  The easy direction.} 
One direction of this theorem is easy.  Since $J_S$ is spanned
by Lyndon insertion shuffles, which lie in $I_S$, we only need to show 
that any Lyndon insertion shuffle is convergent.
If it is a shuffle of convergent Lyndon words,
then there are no consecutive blocks in any of the words.  Therefore
if the letters of any consecutive subset $T$ of $S$ appear as a block
in any term of $\omega$, it must be because they appeared in more than
one of the convergent words which are shuffled together.  So
these letters appear as a shuffle, so the residue lies in
$I_T\otimes {\cal P}_{S\setminus T\cup\{d\}}$, which by definition
\ref{convergence} means that $\omega$ is convergent.
Now, if we are dealing with a Lyndon
insertion shuffle with non-trivial insertions, then there are
two kinds of bad chords: those corresponding to these insertions, and
those corresponding to consecutive subsets of the insertion sets.  For
example, in the Lyndon insertion shuffle
\begin{equation}\label{funnyex}\omega=(2\sha 1(4\sha 35,d)=
2\sha (1435+1345+1354)=
\end{equation}
$$ 21435+12435+14235+14325 +14352+21345+12345+13245+$$
$$13425+13452+ 21354+12354+13254+13524+13542, $$
in which $(4\sha 35)$ is inserted into the Lyndon shuffle $2\sha 13$,
and we write $\omega$ in $V_S$ rather than ${\cal P}_{S\cup \{d\}}$
to avoid adding the index $d$ to the end of every word above.
The bad chord $345$ corresponds to the insertion, and the bad chords
$34$ and $45$ appear in certain terms of the shuffle within the insertion.  
For the latter type, since they appear inside an insertion which is
itself a shuffle, their letters only appear in shuffle combinations
within the insertion (for instance $1435+1345=1(3\sha 4)5$ in the
example above), so the residue along these chords is a shuffle.  But also, 
for the bad chords corresponding to an insertion set, the insertion itself 
lies in ${\cal L}_T\subset I_T$, and is precisely one factor of the residue, 
which is thus also a shuffle.  For example, the residue in the example
above along the chord $E=345$ comes from considering only the terms in
(\ref{funnyex}) which have $\{3,4,5\}$ as a consecutive subset, i.e.
the terms which are polygons admitting the chord $345$, namely 
$$\omega=21435+12435+14352+21345+12345+13452+21354+12354+13542 $$
$$=21(435+345+354)+12(435+345+354)+1(435+345+354)2$$
$$=21(4\sha 35)+ 12(4\sha 35)+1(4\sha 35)2$$
and the residue is thus simply
$${\mathrm Res}_{345}(\omega)=(4\sha 35)\otimes (21e+12e+1e2),$$
where $e$ labels the chord $E$, and the insertion itself is the
left-hand factor. Since insertions always
lie in ${\cal L}_T$, they are always shuffles, therefore $\omega$ converges
along the corresponding chords.
\vskip .3cm
\noindent {\it Step 2. The other direction: Induction hypothesis and
base case.} Assume now that $\omega$ is convergent
and lies in $I_S$, so that
we can write $\omega=\sum_i a_i\omega_i$ where each
$\omega_i=(A^i_1\sha \cdots \sha A^i_{r_i},d)$ is a Lyndon shuffle, $r_i>1$.
We say that a consecutive block appearing in any $A^i_j$ is
maximal if the same block does not appear in that factor or in any other factor
inside a bigger consecutive block.  Factors may appear which contain more than
one consecutive block, but the maximal blocks are disjoint.

We prove the result by induction on the length of the alphabet
$S=\{1,\ldots,n\}$.  The smallest case is $n=3$, since for $n=2$,
the polygons are triangles and have no chords.  For $n=3$, let
$$\omega=c_1(12\sha 3,d)+c_2(13\sha 2,d)+
c_3(1\sha 2\sha 3,d)+c_4(23\sha 1,d)$$
be a linear combination of all the Lyndon shuffles for $n=3$.
The bad chords are $E=\{1,2\}$, $F=\{2,3\}$. We have
$$\Res^p_E(\omega)=c_1(1,2,e)\otimes (e\sha 3,d)
+c_2(1\sha 2,e)\otimes (e,3,d)$$
$$+c_3(1\sha 2,e)\otimes (e\sha 3,d)+c_4(1\sha 2,e)\otimes (e,3,d).$$
For this to converge means that the left-hand parts of the two right-hand
tensor factors $(e,3,d)$ and $(e\sha 3,d)$ must lie in $I_{\{1,2\}}$.
Since three of the four left-hand parts already lie in $I_{\{1,2\}}$,
the fourth one must as well, which must mean that $c_1=0$.  This is
the condition for $\omega$ to converge on $E$.  Now let us consider $F=
\{2,3\}$.  We have
$$\Res^p_F(\omega)=c_1(2\sha 3,f)\otimes (1,f,d)
+c_2(2\sha 3,f)\otimes (1,f,d)$$
$$+c_3(2\sha 3,e)\otimes (1\sha f,d)+c_4(2,3,f)\otimes (1\sha f,d).$$
This gives $c_4=0$ as the condition for $\omega$ to converge on $F$.
Therefore, we find that $\omega$ is a linear combination of $13\sha 2$
and $1\sha 2\sha 3$, which are exactly the elements of the basis
${\cal L}_{\{1,2,3\}}$ of $J_S$. This settles the base case $n=3$.

\vskip .3cm
The induction hypothesis is that for every alphabet
$S'=\{1,\ldots,i\}$ with $i<n$, if $\omega\in I_{S'}$
is convergent, then $\omega\in J_{S'}$.

\vspace{.2cm}
\noindent {\it Step 3. Construction of the insertion terms $(S_{[i]},e)\in
I_T$.}
Now let $S=\{1,\ldots,n\}$ and assume that $\omega\in I_S$ is convergent.
Write $\omega$ as a linear combination of Lyndon shuffles 
$$\omega=\sum_i c_i\omega_i=\sum_ic_i(A^i_1\sha A^i_2\sha
\cdots A^i_{r_i},d).$$
If no consecutive block appears in any $A^i_j$, then $\omega$ is a linear
combination of convergent Lyndon words, so it is in $J_S$ by definition.  
Assume some consecutive blocks do appear, and consider a maximal consecutive 
block $T$, corresponding to a bad chord $E$.  
Decompose $\omega=\gamma_1+\gamma_2$ where $\gamma_k$ is the
sum $\sum_{i\in I_k} c_i\omega_i$, with $I_1$ the set of indices $i$
for which $T$ appears as a block in some $A^i_j$, which by reordering 
shuffled pieces we may
assume to be $A^i_1$, and $I_2$ is the set of indices for which $T$ does not
appear as a block in any $A^i_j$.  Then because letters of $T$ appear
scattered in different $A^i_j$ in each term of $\gamma_2$, any time they
appear as a block in a term of $\gamma_2$, they must appear in several
terms as a shuffle combination, so 
$\Res^p_E(\gamma_2)\in I_T\otimes {\cal P}_{S\setminus T\cup\{e\}\cup \{d\}}$.
Thus $\gamma_2$ converges along $E$. Since we are assuming that $\omega$
is convergent, $\gamma_1$ must then also converge, so we must have
\begin{equation}\label{inshuf}
\Res^p_E(\gamma_1)\in I_T\otimes {\cal P}_{S\setminus T\cup\{d\}\cup \{e\}}.
\end{equation}

For each $i\in I_1$, write $A^i_1=B^i_1Y^iC^i_1$, where $Y^i$ consists of the
letters of $T$ in some order and $B^i_1$ is a (possibly empty) Lyndon word.  

We have
\begin{equation}\label{small}
\Res^p_E(\gamma_1)=\sum_{i\in I_1} c_i(Y^i,e)\otimes (B^i_1eC^i_1\sha
A^i_2\sha\cdots\sha A^i_{r_i},d)
\end{equation}
Note that the alphabet $(S\setminus T)\cup \{e\}$
corresponding to all of the right-hand factors has the lexicographic
ordering inherited from $S$ by deleting the consecutive block
of letters $T$ and replacing it with the unique character $e$.  Thus,
all of the words appearing in the shuffles of the right-hand factors
are Lyndon words.  Indeed, the $A_j^i$, $j>1$, are Lyndon by definition, 
the words $B^i_1eC^i_1$ with non-empty $B^i_1$ are Lyndon because of
the assumption that $A^i_1=B^i_1Y^iC^i_1$ is a Lyndon word and therefore the
smallest character appears on the left of $B^i_1$, and the words $eC^i_1$ 
which appear when $B^i_1$ is empty are Lyndon because $A^i_1=Y^iC^i_1$ is
Lyndon and the characters of $Y^i$ (i.e. those of $T$) are consecutive,
so they are all smaller than those appearing in $C^i_1$; thus $e$ is 
less than any character of $C^i_1$ in the inherited ordering.  Thus, 
all of the right-hand factors of (\ref{small}) are Lyndon shuffles.  

Putting an equivalence relation on $I_1$ by letting
$i\sim i'$ if the right-hand factors of (\ref{small}) are
equal, and letting $[i]$ denote the equivalence classes for this
relation, we write the residue as
\begin{equation}\label{smallbis}
\Res^p_E(\gamma_1)=\sum_{[i]\subset I_1} \bigl(\sum_{i\in [i]} c_i(Y^i,e)\bigr)
\otimes (B_1^{[i]}eC_1^{[i]}\sha A^{[i]}_2\sha\cdots\sha A^{[i]}_{r_{[i]}},d).
\end{equation}
Since the right-hand factors in the sum over $[i]$ are distinct
Lyndon shuffles, the set of right-hand factors forms a linearly independent 
set.  Therefore by (\ref{inshuf}), we must have
\begin{equation}\label{defins}(S_{[i]},e)=\sum_{i\in [i]} c_i(Y^i,e)\in I_T
\end{equation}
for each $[i]\subset I_1$.

Let us show that $(S_{[i]},e)=0$ whenever $B^{[i]}_1$ is empty.
For all $i\in I_1$ such that $B^i_1$ is empty, we have $A^i_1=Y^iC^i_1$, 
and since these are all Lyndon words, the smallest character of $T$, 
say $a$, is always on the left of $Y^i$, so we can write $Y^i=aY^i_0$ and 
$A^i_1=aY_0^iC^i_1$ for all such $i$.  Then for an equivalence class
$[i]$ of such $i$, the $(S_{[i]},e)$ of (\ref{defins}) can be written
$$(S_{[i]},e)=\sum_{i\in [i]} c_i(Y^i,e)=\sum_{i\in [i]} c_i(aY_0^i,e)\in I_T.$$
But by lemma \ref{usefullemma}, a sum of words all having the same character
(here $a$) on the left and the same character (here $e$) on the right 
cannot be a shuffle unless it is zero, so $(S_{[i]},e)=0$ if $B^{[i]}_1$
is empty.

\vspace{.2cm}
\noindent {\it Step 4. Proof that the insertion terms 
$(S_{[i]},e)$ lie in $J_T$.}  
For this, we first need to show that $(S_{[i]},e)$ converges on every
subchord of $E$, i.e. every consecutive subset inside the set $T$, before
applying the induction hypothesis.  Let $E'$ be a subchord of $E$, 
corresponding to a consecutive block $T'$ strictly contained in $T$.

Decompose the set of indices $I_1$ into two subsets $I_3$ and
$I_4$, where $I_3$ contains the indices $i\in I_1$ such that $T'$ appears
as a consecutive block inside the block $T$ appearing in $A^i_1$,
and $I_4$ contains the indices $i\in I_1$ such that the letters of $T'$ do not
appear consecutively inside the block $T$. Similarly, partition $I_2$,
the set of indices in the sum $\omega=\sum_i c_i\omega_i$ for
which $T$ does not appear as a block in $A^i_1$,
into two sets $I_5$ and $I_6$, where $I_5$ contains the indices $i\in I_2$
such that $T'$ appears as a block in some $A^i_j$ which we may assume
to be $A^i_1$, and $I_6$ contains the indices $i\in I_2$ of the terms in which
$T'$ does not appear as a block in any $A^i_j$.  We have corresponding
decompositions $\gamma_1=\gamma_3+\gamma_4$, $\gamma_2=\gamma_5+\gamma_6$.

As before, $T'$ must appear as a shuffle in $\gamma_6$, so $\gamma_6$
converges along $E'$.  As for $\gamma_4$, since $T'$ does not appear as either
a block or a shuffle, the residue along $E'$ is $0$.  Since by assumption
$\omega=\gamma_3+\gamma_4+\gamma_5+\gamma_6$ converges along $E'$, we see that
$\gamma_3+\gamma_5$ must converge
along $E'$.  Let us show that in fact both $\gamma_3$ and $\gamma_5$
converge along $E'$.

Write $A^i_1=R^iZ^iS^i$ for every $i\in I_3\cup I_5$, where $Z^i$ is a word
in the letters of $T'$.  Note that $R^i$ is Lyndon, and non-empty by the
identical reasoning to that used above to show that $B^i_1$ is non-empty.
Then for $k=3,5$, we have
\begin{equation}\label{small2}
\Res^p_{E'}(\gamma_k)=\sum_{i\in I_k} c_i(Z^i,e')\otimes (R^ie'S^i\sha
A^i_2\sha\cdots\sha A^i_{r_i},d).
\end{equation}
For $k=3,5$, put the equivalence relation on $I_k$ for which
$i\sim  i'$ if the right-hand factors of (\ref{small2}) are equal, and
let $\langle i\rangle$ denote the equivalence classes for this relation.
Note that because for $i\in I_3$, $T'$ appears as a block of $T$,
the word $B^i_1$ must appear as the left-hand part of $R^i$, and the word
$C^i_1$ must appear as the right-hand part of $S^i$.
Therefore, in particular, the new equivalence relation
is strictly finer than the old, i.e. the equivalence class $[i]$ breaks
up into a finite union of equivalence classes $\langle i\rangle$.
The residues for $k=3,5$ can now be written
\begin{equation}
\Res^p_{E'}(\gamma_k)=\sum_{\langle i\rangle\subset I_k}
\bigl(\sum_{i\in\langle i\rangle} c_i(Z^i,e')
\bigr) \otimes (R^{\langle i\rangle }e'S^{\langle i\rangle}\sha
A^{\langle i\rangle }_2 \sha\cdots\sha
A^{\langle i\rangle}_{r_{\langle i\rangle}}).
\end{equation}
Then since the right-hand factors for each $k$ are distinct Lyndon shuffles,
they are linearly independent. Furthermore, none of the right-hand
factors occurring in the sum for $k=3$
can ever occur in the sum for $k=5$ for the following reason:
the Lyndon words $R^ie'S^i$ appearing for $k=3$ all have the letters of
$T\setminus T'$ grouped around $e'$, whereas none of the Lyndon words
$R^ie'S^i$ have this property.  Therefore all the right-hand
factors from the residues of $\gamma_3$ and $\gamma_5$ together form
a linearly independent set, so we find that all the left-hand factors
\begin{equation}\label{small3}
\sum_{i\in \langle i\rangle \subset I_k} (Z^i,e')\in I_{T'},
\end{equation}
so that both $\gamma_3$ and $\gamma_5$ converge along $E'$.  In
particular, this means that both $\gamma_1$ and $\gamma_2$ converge
along $E'$.

Now, to determine that the $(S_{[i]},e)$ of (\ref{defins}) converge
along $E'$, we will use (\ref{smallbis}) to compute the composed residue 
map $\Res^p_{E,E'}(\gamma_1)$.  We are only concerned with the set
of indices $I_1=I_3\cup I_4$ in (\ref{smallbis}).
For each $i\in I_3$, write $Y^i=U^iZ^iV^i$ where $Z^i$ is a word in the
letters of $T'$, so that $R^i=B^iU^i$, $S^i=V^iC^i$, and
$A^i_1=B^iU^iZ^iV^iC^i$.
Then by (\ref{small2}), we have
$$\Res^p_E(\gamma_1)=\sum_{[i]\in I_3}\biggl(\sum_{i\in [i]} c_i(U^iZ^iV^i,e)
\biggr)\otimes \bigl(B^{[i]}_1eC^{[i]}_1\sha A^{[i]}_2 \sha  \cdots \sha
A^{[i]}_{r_{[i]}},d\bigr)+$$
$$\sum_{[i]\in I_4}\biggl(\sum_{i\in [i]} c_i(Y^i,e)\biggr)\otimes
\bigl(B^{[i]}_1eC^{[i]}_1\sha A^{[i]}_2 \sha  \cdots \sha
A^{[i]}_{r_{[i]}},d\bigr).  $$
The terms for $i\in I_4$ converge along $T'$, so they vanish when taking
the composed residue, and we find
$$\Res^p_{E,E'}(\gamma_1)=
\sum_{[i]\in I_3} \biggl(\sum_{i\in [i]} c_i(Z^i,e')
\otimes (U^ie'V^i,e)\biggr)\otimes
\bigl(B^{[i]}_1eC^{[i]}_1\sha A^{[i]}_2 \sha  \cdots \sha
A^{[i]}_{r_{[i]}},d\bigr).$$
Since for each $[i]\subset I_3$, the right-hand factors are as usual distinct
and linearly independent, this means that for each $[i]\subset I_3$,
$$\Res^p_{E'}(S_{[i]},e)=\sum_{i\in [i]}
c_i(Z^i,e')\otimes (U^ie'V^i,e)\in {\cal P}_{T'\cup \{e'\}}\otimes
{\cal P}_{T\setminus T'\cup\{e'\}\cup \{e\}}.$$
Now, the equivalence relation on $i\in [i]\subset I_3$ given by
$i\sim i'$ if $U^i=U^{i'}$ and $V^i=V^{i'}$ is the same as the 
equivalence relation $i \sim i'$ if $R^i=R^{i'}$ and $S^i=S^{i'}$ since 
$R^i=B^iU^i$ and $S^i=V^iC^i$.  So the classes $\langle i\rangle$ correspond
to sets of $i$ for which $U^i$ and $V^i$ are identical.
Thus for each $[i]\subset I_3$, we can write
$$\Res^p_{E'}(S_{[i]},e)=\sum_{\langle i\rangle\subset [i]}
\biggl(\sum_{i\in \langle i\rangle} c_i(Z^i,e')\biggr)
\otimes (U^{\langle i\rangle}e'V^{\langle i\rangle},e),$$
where the right-hand factors are all distinct words.
Then (\ref{small3}) shows that this sum lies in
$I_{T'}\otimes {\cal P}_{T\setminus T'\cup \{e'\}\cup \{e\}}$, so in fact
$(S_{[i]},e)$ converges along $E'$.  For $[i]\subset I_4$, we have
saw that $\Res^p_{E'}((S_{[i]},e))=0$ since $T'$ never occurs as a block
for $i\in I_4$.  Thus $(S_{[i]},e)$ converges along $E'$
for all $[i]\subset I_1$.  

Since we have just shown that $(S_{[i]},e)$ converges along every
subchord $E'$ of $E$, i.e. along the chords corresponding to every
consecutive subblock $T'$ of $T$, we see that each term
$(S_{[i]},e)$ is convergent along all its bad chords.
Thus, by the induction hypothesis, $(S_{[i]},e)\in J_T$.

\vspace{.2cm}
\noindent {\it Step 5. Construction of the insertions.}
The above construction shows that we can write $\omega=\gamma_1+\gamma_2$ with
$$\gamma_1=\sum_{[i]\in I_1}  c_{[i]} \bigl(B^{[i]}S_{[i]}C^{[i]}\sha
A^{[i]}_2\sha \cdots \sha A^{[i]}_{r_i},d\bigr)$$
with $S_{[i]}\in J_T$.  This means that the maximal block $T$, which appeared
only in $\gamma_1$, has been replaced by an insertion in the sense of the
definition of Lyndon insertion shuffles.
To conclude the proof of the theorem, we successively replace each of the
maximal blocks in $\omega$ by insertion terms in the same way, in any
order, since maximal blocks are disjoint.  
The final result displays $\omega$ as a linear combination of convergent
Lyndon shuffles and Lyndon insertion shuffles, so $\omega\in J_S$.
\end{proof}

The following theorem is the exact analogy of the previous one, but with
the actual shuffles in $I_S$ replaced by the words in $W_S$ that have
$1$ just to the left of $n$, and the set of Lyndon insertion shuffles
replaced by Lyndon insertion words, which considerably simplifies the proof.
\begin{thm}\label{convKS}
Let $\eta\in W_S\subset {\cal P}_{S\cup \{d\}}$.
Then $\eta$ is convergent if and only if $\eta\in K_S=\langle
{\cal W}_S\rangle$.
\end{thm}
\begin{proof} The proof that $\omega\in K_S$ is convergent is exactly
as at the beginning of the proof of the previous theorem.  So consider
the other direction.  Let $\omega\in W_S$, so that we can write
$$\omega=\sum_i a_i\eta_i$$
where each $\eta_i$ is a $1n$-polygon (a $1n$-word concatenated with $d$),
and assume $\omega$ is convergent.
The only possible bad chords for $\omega$ are the consecutive blocks
appearing in the $\eta_i$.  Let $T$ be a subset of $S$ corresponding
to a maximal consecutive block.
\vskip .3cm
\begin{lem}\label{lemleft}No maximal consecutive block
having non-trivial intersection with $\{1,n\}$ can appear in
any of the $1n$-words $\eta_i$ of $\omega$.
\end{lem}
\begin{proof}  If $T$ is a maximal block containing both $1$ and $n$,
then $T=\{1,\ldots,n\}$ which does not correspond to a chord.

Assume now that $T=\{m,\ldots,n\}$ with $m>1$.  Write 
$\eta_i=(K^i,1,n,Z^i,H^i,d)$ where $Z^i$ is an ordering of 
$\{m,\ldots,n-1\}$.  Let $E$ be the chord corresponding to $T$.
We have
$$\Res^p_E(\sum_i a_i\eta_i)=\sum_i a_i(n,Z^i,e)\otimes (K^i,1,e,H^i,d).$$
Convergence implies that for any constant words $K$, $H$, the sum
\begin{equation}
\sum_{i\mid K^i=K,H^i=H} a_i(n,Z^i,e)\in I_T.
\end{equation}
But by lemma \ref{usefullemma}, it is impossible for a sum of words 
all having the same character on the left to be equal to a shuffle.

The case where $T=\{1,\ldots,m\}$ with $m<n$ is identical, except for
an easy adaptation of lemma \ref{usefullemma} to show that a sum of words 
all having the same character on the right cannot be equal to a shuffle.
\end{proof}

Now we can complete the proof of the theorem.  
Let $\omega=\sum_i a_i\eta_i$ be
a sum of $1n$-words which converges, and consider a maximal consecutive
block $T\subset \{2,\ldots,n-1\}$.  Let $I_1$ be the set of indices
$i$ such that $\eta_i$ contains the block $T$ and $I_2$ the other indices.
For $i\in I_1$, write $\eta_i=(K^i,Z^i,H^i,d)$ where $Z^i$ is an ordering of
$T$.  Then
$$\Res^p_T(\omega)=\sum_{i\in I_1} a_i(Z^i,e)\otimes (K^i,e,H^i,d).$$
Let $i\sim i'$ be the equivalence relation on $I_1$ given by
$K^i=K^{i'}$ and $H^i=H^{i'}$.  Then
$$\Res^p_T(\omega)=\sum_{[i]\in I_1} \bigl(\sum_{i\in [i]}
a_i(Z^i,e)\bigr)\otimes (K^{[i]},e,H^{[i]},d),$$
and the right-hand factors are all distinct (linearly independent) words, so
by the assumption that $\omega$ convergence along $E$, we have
$$(S_{[i]},e)=\sum_{i\in [i]} a_i(Z^i,e)\in I_T$$
for each $[i]\subset I_1$.  Therefore we can write $\omega$ as
$$\omega=\sum_{[i]\subset I_1} a_i(K^{[i]},S_{[i]},H^{[i]},d)+
\sum_{i\in I_2} a_i\eta_i,$$
with the maximal block $T$ replaced by the insertion $S_{[i]}$.
We prove that $S_{[i]}\in J_T$  exactly as in the proof of the previous
theorem: considering a maximal consecutive block $T'\subset T$ occurring
in a factor of $S_{[i]}$, one shows that $S_{[i]}$ converges along $T'$
if and only if $\omega$ converges along $T'$.  Since $\omega$ does
converge by assumption, $S_{[i]}$ also converges, and since this holds
for all consecutive blocks $T'\subset T$, $S_{[i]}$ converges on all
its subdivisors and therefore $S_{[i]}\in J_S=\langle {\cal L}_S\rangle$.
Finally, one deals with the disjoint maximal blocks appearing in $\omega$ one
at a time until no blocks at all remain, expressing $\omega$ explicitly
as a linear combination of Lyndon insertion words.
\end{proof}

\vspace{.3cm}
\noindent {\bf A summary of the results in this chapter.} We introduced
the following spaces, where $S=\{1,\ldots,n\}$:
\begin{itemize}
\item $V_S$: the $\Q$-vector space generated by words in $S$ having distinct
letters 
\item $I_S$: the $\Q$-vector space generated by shuffles of disjoint words of $V_S$ (definition \ref{IS})
\item ${\cal L}_S$: the set of Lyndon insertion shuffles (definition \ref{defngrr}),
which are linearly independent (theorem \ref{linind})
\item $J_S$: the subspace of $I_S$ spanned by ${\cal L}_S$, which forms the 
set of convergent elements of $I_S$ (theorem \ref{convJS})
\item $W_S$: the $\Q$-vector space generated by words in $V_S$, so that by
Radford's theorem, we have $V_S=I_S\oplus W_S$
\item ${\cal W}_S$: the set of Lyndon insertion words (definition \ref{insw}), which
are linearly independent (theorem \ref{linind})
\item $K_S$: the subspace spanned by ${\cal W}_S$, which forms the set of
convergent elements of $W_S$ (theorem \ref{convKS}).
\end{itemize}
\vspace{.3cm}
\section{Explicit generators for ${\cal FC}$ and $H^\ell(\Mod_{0,n}^\delta)$}

In this chapter, we show that the map from polygons to cell-forms is
surjective, and compute its kernel.  From this and the previous
chapter, we will conclude that the pairs $(\delta,\omega)$, where
$\omega$ runs through the set ${\cal W}_S$ of Lyndon insertion words
for $n\ge 5$, form a generating set for the formal cell-zeta algebra
${\cal FC}$. In the final section, we show that the images of the
elements of ${\cal W}_S$ in the cohomology $H^\ell(\Mod_{0,n})$
yield an explicit basis for the convergent cohomology
$H^\ell(\Mod_{0,n}^\delta)$, determine its dimension, and compute the
cohomology basis explicitly for small values of $n$.  We recall that
$\Mod_{0,n}^\delta$ is defined in section \ref{thing}, and that by
the ``convergent cohomology'', we mean the cohomology classes of $\ell$-forms
with logarithmic singularities which converge on the closure of the 
standard cell.

\subsection{From polygons to cell-forms}\label{sec41}
Let $S=\{1,\ldots,n\}$.  The bijection $\rho:S\cup \{d\}\rightarrow
\{0,t_1,\ldots,t_{\ell+1},1,\infty\}$ given by associating the
elements $1,\ldots,n,d$ to $0,t_1,\ldots,t_{\ell+1}, 1,\infty$
respectively, induces a  map $f$ from polygons to cell-forms:
$$\eta=(\sigma(1),\ldots,\sigma(n),d)\buildrel{f}\over\rightarrow
\omega_\eta=[\rho(\sigma(1)),\ldots,\rho(\sigma(n)),\infty].$$ The
map $f$ extends by linearity to a map from ${\cal P}_{S\cup \{d\}}$
to the cohomology group $H^{n-2}(\Mod_{0,n+1})$.  The purpose of
this section is to prove that $f$ is a surjection, and to determine
its kernel.

Recall that $I_S\subset {\cal P}_{S\cup \{d\}}$ denotes the subvector space of
${\cal P}_{S\cup \{d\}}$ spanned by the {\it shuffles with respect to the
element $d$}, namely by the linear combinations of polygons
$$(S_1 \sha S_2, d)$$
for all partitions $S_1\coprod S_2$ of $S$.

\begin{prop}\label{prop41} Let $S=\{1,\ldots,n\}$.  Then the cell-form map
$$f:{\cal P}_{S\cup \{d\}} \To H^{n-2}(\Mod_{0,n+1})$$
is surjective with kernel equal to the subspace $I_S$.
\end{prop}

\begin{proof} The surjectivity is an immediate consequence of the
fact that $01$ cell-forms form a basis of $H^{n-2}(\Mod_{0,n+1})$
(theorem \ref{thm01cellsspan}), since all such cell-forms are the
images under $f$ of polygons having the edge labelled $1$ next to
the one labelled $n$.

Now, $I_S$ lies in the kernel of $f$ by the corollary to proposition
\ref{shufprod}. So it only remains to show that the kernel of $f$ is
equal to $I_S$.  But this is a consequence of counting the
dimensions of both sides.  By theorem \ref{thm01cellsspan}, we know
that the dimension of $H^{n-2}(\Mod_{0,n+1})$ is equal to $(n-1)!$.
As for the dimension of ${\cal P}_{S\cup \{d\}}/I_S$, recall from
the beginning of chapter 3 that ${\cal P}_{S\cup \{d\}}\simeq V_S$,
which can be identified with the graded $n$ part of the quotient of
the polynomial algebra on $S$ by the relation $w=0$ for all words
$w$ containing repeated letters.  Thus $V_S$ is the vector space
spanned by words on $n$ distinct letters, so it is of dimension
$n!$. But instead of taking a basis of words, we can take the Lyndon
basis of Lyndon words (words with distinct characters whose smallest
character is on the left) and shuffles of Lyndon words.  The
subspace $I_S$ is exactly generated by the shuffles, so the
dimension of the quotient is given by the number of Lyndon words on
$S$, namely $(n-1)!$. Therefore ${\cal P}_{S\cup \{d\}}/I_S\simeq
H^{n-2}(\Mod_{0,n+1})$.
\end{proof}

\begin{rem}
The above proof has an interesting consequence.  Since the map from
polygons to differential forms does not depend on the role of $d$,
the kernel cannot depend on $d$, and any other element of $S\cup
\{d\}$ could play the same role.  Therefore $I_S$, which is defined
as the space generated by shuffles with respect to the element $d$,
is equal to the space generated by shuffles of elements of $S\cup
\{d\}$ with respect to any element of $S$; it is simply the subspace
generated by {\it shuffles with respect to one element} of $S\cup
\{d\}$.
\end{rem}

\begin{cor}
Let $W_S\subset {\cal P}_{S\cup \{d\}}$ be the subset of polygons
corresponding to $1n$-words (concatenated with $d$).  Then
$$f:W_S\simeq H^{n-2}(\Mod_{0,n+1}).$$
\end{cor}

\begin{proof} The proof follows from the fact that ${\cal P}_{S\cup \{d\}}=
W_S\oplus I_S$.
\end{proof}

\subsection{Generators for ${\cal FC}$}

By definition, ${\cal FC}$ is generated by all linear combinations
of pairs of polygons $\sum_i a_i (\delta,\omega_i)$ whose associated
differential form converges on the standard cell, but modulo the
relation (among others) that shuffles are equal to zero.  In other
words, since ${\cal P}_{S\cup \{d\}}=W_S\oplus I_S$, we can redefine
${\cal FC}$ to be generated by linear combinations $\sum_i
a_i(\delta,\omega_i)$ such that $\sum_i a_i\omega_i\in W_S$ and such
that the associated differential form converges on the standard
cell.

The following proposition states that the notion of the residue of a
polygon and the residue of the corresponding cell-form coincide. In
order to state it, we must recall that one can define the map
$$\rho: \cal{P}_S \To \Omega^\ell(\Mod_{0,S})\ ,$$
from polygons labelled by $S$ to cell forms in a coordinate-free way
(one can do this directly from equation $(\ref{omegalift})$). In
$\S1$, this map was defined in explicit coordinates by fixing any
three marked points at $0,1$ and $\infty$. This essence of lemma
$\ref{lemsymaction}$ is that $\rho$ is independent of the choice of
three marked points, and is thus  coordinate-free.

\begin{prop} \label{propresformula}  Let $S=\{1,\ldots,n\}$ and let
$D$ be a stable partition $S_1\cup S_2$ of $S$ corresponding to a boundary
divisor of $\Mod_{0,n}$, with $|S_1|=r$ and $|S_2|=s$.  Let $\rho$ denote
the usual map from polygons to cell-forms.  Then the following
diagram is commutative:
$$\xymatrix{{\cal P}_S\ar[r]^\rho\ar[d]_{\Res^p_D} &H^\ell(\Mod_{0,n})
\ar[d]^{\Res_D}\\
{\cal P}_{S_1\cup \{d\}}\otimes  {\cal P}_{S_2\cup \{d\}}\ar[r]^{\!\!\!
\!\!\!\!\!\!\!\!\!\!\!\!\!\!\!\!\!\!\rho \otimes\rho}&\
H^{r-2}(\Mod_{0,r+1})\otimes H^{s-2}(\Mod_{0,s+1}).  }$$
In other words, the usual residue of differential forms
corresponds to the combinatorial residue of polygons.
\end{prop}
\begin{proof}
 Let $\eta\in {\cal P}_S$ be a polygon, and let $\omega_\eta$
be the associated cell-form.  If $D$ is not compatible with
$\omega_\eta$, then $\omega_\eta$ has no pole on $D$ by proposition
\ref{CORomegapoles}, so $\Res_D(\omega)=0$.

  We shall
work in explicit coordinates, bearing in mind that this does not
affect the answer, by the remarks above.
%
 Therefore assume that $\eta$ is the polygon
numbered with the standard cyclic order on $\{1,\ldots,n\}$, and
that $D$ is compatible with $\eta$.  The corresponding cell-form is
given in simplicial coordinates by $[0,t_1,\ldots,t_\ell,1,\infty]$.
By applying a cyclic rotation, we can assume that $D$ corresponds to
the partition
$$S_1 = \{1,2,3,\ldots, k+1\} \ \hbox{ and } \ \
S_2=\{k+2,\ldots, n-1, n\}$$ for some $1\leq k\leq \ell$.
 In simplicial coordinates, $D$ corresponds to the  blow-up of the
cycle $0=t_1=\cdots=t_k$. We compute the residue of $\omega_\eta$
along $D$ by applying the variable change $t_1  = x_1\ldots x_\ell,
\ldots,  t_{\ell-1} = x_{\ell-1}x_\ell,  t_\ell = x_\ell$ to the
form $\omega_\eta=[0,t_1,\ldots,t_\ell,1,\infty]$. The standard cell
$X_\eta$ is given by $\{0<x_1,\ldots,x_\ell<1\}$. In these
coordinates, the divisor $D$ is given by $\{x_k=0\}$, and the form
$\omega_\eta$ becomes
\begin{equation} \label{omegacubicalform}
\omega_{\eta} = {dx_1\ldots dx_\ell \over x_1(1-x_1)\ldots
x_\ell(1-x_\ell)}.  \end{equation} The residue of $\omega_\eta$
along $x_k=0$ is given by
\begin{equation}\label{proofres}
{dx_1\ldots dx_{k-1} \over x_1(1-x_1)\ldots x_{k-1}(1-x_{k-1}) }
\otimes {dx_{k+1}\ldots dx_{\ell} \over x_{k+1}(1-x_{k+1})\ldots
x_{\ell}(1-x_{\ell})}\ .\end{equation} Changing back to simplicial
coordinates via
  $x_1=a_1/a_2,\ldots,x_{k-2}=a_{k-2}/a_{k-1}$,
$x_{k-1}=a_{k-1}$,  and $x_\ell=b_\ell$,
$x_{\ell-1}=b_{\ell-1}/b_\ell, \ldots,x_{k+1}=b_k/b_{k+1}$ defines
simplicial coordinates on $D\cong \Mod_{0,r+1}\times \Mod_{0,s+1}$.
The standard cells induced by $\eta$ are
$(0,a_1,\ldots,a_{k-1},1,\infty)$ on $\Mod_{0,r+1}$ and
$(0,b_k,\ldots, b_{\ell},1,\infty)$ on $\Mod_{0,s+1}$. If we compute
$(\ref{proofres})$ in these new coordinates, it gives precisely
$$ [0,a_1,\ldots,a_{k-1},1,\infty]\otimes [0,b_k,\ldots,b_\ell,1,\infty]\ ,$$
which is the tensor product of the cell forms corresponding to the
standard cyclic orders $\eta_1,\eta_2$ on $S_1\cup\{d\}$ and
$S_2\cup\{d\}$ induced by $\eta$. Therefore $\rho(\Res^p_D
\eta)=\Res_D \omega_\eta$.

To conclude the proof of the proposition, it is enough to notice
that applying $\sigma\in \Sym(n)$ to the formula
$\hbox{Res}_D\omega_\eta=\omega_{\eta_1}\otimes \omega_{\eta_2}$
yields
$$\hbox{Res}_{\sigma(D)}\sigma^*(\omega_\eta)=
\hbox{Res}_{\sigma(D)}\omega_{\sigma(\eta)}=\sigma^*(\omega_{\eta_1})
\otimes \sigma^*(\omega_{\eta_2})=\omega_{\sigma(\eta_1)}\otimes
\omega_{\sigma(\eta_2)}.$$
Here, 
$\sigma(\eta_i)$ is the  cyclic order induced by $\sigma(\eta)$ on
the set $\sigma(S_1)\cup \{\sigma(d)\}$, where $\sigma(d)$
corresponds to the partition $S=\sigma(S_1)\cup \sigma(S_2)$. Thus
$\rho(\Res^p_{\sigma(D)} \sigma(\eta))=\Res_{\sigma(D)}
\omega_{\sigma(\eta)}$ for all $\sigma \in \Sym(n)$, which proves
that $\rho(\Res^p_D \gamma)=\Res_D \omega_\gamma$ for all cyclic
structures $\gamma\in\cal{P}_S$, and all divisors $D$.
\end{proof}

\begin{cor} A linear combination $\eta=\sum_i a_i\eta_i\in W_S
\subset {\cal P}_{S\cup \{d\}}$ converges with respect to the standard polygon
if and only if its associated form $\omega_\eta$ converges on the standard cell.
\end{cor}

\begin{proof}
We first show that \begin{equation}\label{rescond}\Res^p_D(\eta)\in
I_{S_1}\otimes {\cal P}_{S_2\cup \{d\}}+ {\cal P}_{S_1\cup
\{d\}}\otimes I_{S_2}\ ,\end{equation} if and only if $\omega_\eta$
converges along the corresponding divisor $D$ in the boundary of the
standard cell. If $(\ref{rescond})$ holds,
then by proposition \ref{prop41} together with the previous
proposition, $\Res_D(\omega_\eta)=0$. Conversely, if
$\Res_D(\omega_\eta)=0$ for a divisor $D$ in the boundary of the
standard cell, then by the previous proposition, $\Res^p_D(\eta)\in$
Ker$(\rho\otimes\rho)$, which is exactly equal to $I_{S_1}\otimes
{\cal P}_{S_2\cup \{d\}}+ {\cal P}_{S_1\cup \{d\}}\otimes I_{S_2}$.

We now show that $(\ref{rescond})$ is equivalent to the convergence
of $\eta$. But since $\eta\in W_S$, the argument of lemma
$\ref{lemleft}$ implies that $(\ref{rescond})$ holds automatically
for any $D$ which intersects $\{1,n\}$ non-trivially. If $D$
intersects $\{1,n\}$ trivially, then we can assume that $\{1,n\}
\subset S_2$. In that case, the fact that $W_{S_2}\cap I_{S_2}=0$
(lemma \ref{usefullemma}) implies that $(\ref{rescond})$ is
equivalent to the apparently stronger condition
$$\Res^p_D(\eta)\in
I_{S_1}\otimes {\cal P}_{S_2\cup \{d\}}\ ,$$ and thus $\eta$
converges along $S_1$ in the sense of definition (\ref{conv}). This
holds for all divisors $D$ and thus completes the proof of the
corollary.
\end{proof}

\begin{cor} The Lyndon insertion words of ${\cal W}_S$ form a generating
set for ${\cal FC}$.  Furthermore, ${\cal FC}$ is defined by subjecting
this generating set to only two sets of relations (cf. definition \ref{threerels})
\begin{itemize}
\item{dihedral relations}
\item{product map relations}
\end{itemize}
\end{cor}

\begin{rem}The third relation from definition \ref{threerels} is not needed
because we have restricted attention from all linear combinations of pairs
of polygons to only those in the basis ${\cal W}_S$, where such shuffles
do not occur. 
\end{rem}

\vspace{.3cm}
\subsection{The insertion basis for $H^\ell(\Mod_{0,n}^\delta)$}\label{sec43}

\begin{defn}\label{definsforms} Let an {\it insertion form} be the sum of $01$-cell forms
obtained by renumbering the Lyndon insertion words of ${\cal W}_S$ via
$(1,\ldots,n,d)\rightarrow (0,t_1,\ldots,t_{\ell+1},1,\infty)$.
\end{defn}

\begin{thm}\label{insformsspan} The insertion forms form a basis for
$H^{n-2}(\Mod_{0,n+1}^\delta)$.
\end{thm}

This is an immediate corollary of all the preceding results.

It is interesting to attempt to determine the dimension of the
spaces $H^\ell(\Mod_{0,n}^\delta)$.  The most important numbers
needed to compute these are the numbers $c_0(n)$ of special
convergent words (convergent 01 cell-forms) on $\Mod_{0,n}$.  
These can be computed by counting the number of polygons indexed by
symbols $(0,t_1,\ldots,t_\ell,1,\infty)$ (or $(1,\ldots,n)$)
which are convergent with respect to the standard cyclic order and
also have the index $0$ next to $1$ (or $1$ next to $n-1$); in other words, 
the number of cyclic orders having $0$ next to $1$ and in which no
$k$ consecutive labels occur as a single block of $k$ consecutive elements
of the cyclic order.  By direct counting, we
find $c_0(4)=0$, $c_0(5)=1$, $c_0(6)=2$, $c_0(7)=11$, $c_0(8)=64$,
$c_0(9)=461$.

\begin{prop}
Set $I_1=1$, and let $I_r$ denote the cardinal of the set
${\cal L}_{\{1,\ldots,r\}}$ for $r\ge 2$ given in definition \ref{defngrr}.  
The dimensions $d_n=$dim$\,H^\ell(\Mod_{0,n}^\delta)$ are given by
\begin{equation}\label{diminsertionformula}
d_n = \sum^n_{r=5} \ \ \ \sum_{i_1+\cdots + i_{r-3}= n-3} I_{i_1}\ldots
I_{i_r} c_0(r)\ , \end{equation} where the inner sum is over all
partitions of $(n-3)$ into $(r-3)$ strictly positive integers. This formula
can be written as follows in terms of generating series.  Let
$I(x)=\sum_{n=1}^\infty I_nx^n=x+x^2+2x^3+7x^4+\cdots$, and let
$C(x)=\sum_{r=5}^\infty c_0(r)x^{r-3} = x^2+2x^3+11x^4+64x^5+\cdots$.
Then if $D(x)=\sum_{n=5}^\infty d_nx^{n-3}$, we have the identity
$$D(x)=C\big(I(x)\bigr).$$
\end{prop}

\begin{proof} This recursive counting formula is a direct consequence
of the definition, counting all possible ways of making
insertions into the $c_0(r)$ convergent $01$-cell forms for $5\le r\le n$.
\end{proof}
\begin{rem}
We have $I_1=I_2=1$, $I_3=2$, $I_4=7$, $I_5=34$, $I_6=206$ (see example
\ref{examplegrr}).  The formula gives
\begin{equation*}
\begin{cases}
d_5=I_1^2c_0(5)=1\ ,\\
d_6=I_1I_2c_0(5)+I_2I_1c_0(5)+I_1^3c_0(6)=1+1+2=4\ ,\\
d_7=I_1I_3c_0(5)+I_2^2c_0(5)+I_3I_1c_0(5)+I_1^2I_2c_0(6)+I_1I_2I_1c_0(6)
+I_2I_1^2c_0(6)+c_0(7)\\
\ \ =5c_0(5)+3c_0(6)+c_0(7)=5+6+11=22\ .
\end{cases}
\end{equation*}
The authors thank Don Zagier for the restatement of formula 
(\ref{diminsertionformula}) in terms of generating series.  
In the forthcoming preprint \cite{BB}, the following remarkably
simple identity concerning the $d_n$ is proven.
Let $E(x)=x-x^2-\sum_{n=4}^\infty d_nx^{n-1}$, and set
$F(x)=\sum_{n=1}^\infty (n-1)! x^n$.  Then
$$E\bigl(F(x)\bigr)=x,$$
in other words $E(x)$ is the formal inversion of the power series $F(x)$. 

While the present paper was in the final stages of correction, 
a preprint \cite{ST} appeared in which a sequence of numbers $d_n$,
of which the first ones are equal to the $d_n$ defined above, are
discovered and interpreted in terms of free Lie operads.
In this paper, the authors give the same expression for the generating
series of their $d_n$ as the inverse of $F(x)$, thus their result
provides a new interpretation of the dimensions $d_n$.

Note that the formula (\ref{diminsertionformula}) gives the dimensions 
as sums of positive terms. A very different  formula for 
dim$\,H^\ell(\Mod_{0,n}^\delta)$ is given in \cite{BB} using point-counting 
methods.  The relations between the proof in \cite{ST}, the geometry of moduli 
spaces, the intermediate power series $I(x)$ and $C(x)$, and 
the counting method in \cite{BB}, will be discussed in
a forthcoming paper.

\end{rem}

\vspace{.3cm}
\subsection{The insertion basis for $\Mod_{0,n}$, $5\le n\le 9$}
 \label{calculations} In this section we list the insertion bases
in low weights.  In the case $\Mod_{0,5}$, there is a single
convergent cell form: \begin{equation}\label{insbasis5}
\omega=[0,1,t_1,\infty,t_2].
\end{equation}
The corresponding period integral is the cell-zeta value:
$$\zeta(\omega) = \int_{(0,t_1,t_2,1,\infty)} [0,1,t_1,\infty,t_2] =
\int_{0\leq t_1\leq t_2\leq 1}  {dt_1 dt_2 \over (1-t_1) t_2 } =
\zeta(2)\ .$$ Here we use the notation of round brackets for cells
in the moduli space $\Mod_{0,n}$ introduced in section
\ref{prodmaps}: the cell $(0,t_1,t_2,1,\infty)$ is the same as the
cell $X_{5,\delta}$ corresponding to the standard dihedral order on
the set $\{0,t_1,t_2,1,\infty\}$. Since $C_0(5)$ is 1-dimensional,
the space of periods in weight $2$, namely the weight 2 graded part
${\cal C}_2$ of the algebra of cell-zeta values ${\cal C}$ of
section \ref{cellalg}, is just the 1-dimensional space spanned by
$\int_{X_{5,\delta}} \omega=\zeta(2)$.

\subsubsection{The case $\Mod_{0,6}$}
The space $C(6)$ is four-dimensional, generated by two
$01$-convergent cell-forms (the first row in the table below) and
two forms (the second row in the table below) which come from
inserting $\cal{L}_{1,2}=\{1\sha2\}$ and $\cal{L}_{2,3}=\{2\sha 3\}$
into the unique convergent $01$ cell form on $\Mod_{0,5}$
(\ref{insbasis5}). The position of the point $\infty$ plays a
special role. It gives rise to another grading, corresponding to the
two columns in the table below, since $\infty$ can only occur in two
positions. \vspace{0.1in}
\begin{center}
\begin{tabular}{|c|c|c|}
  \hline
$C_0(6)$ & $\omega_{1,1}=[0,1,t_2,\infty,t_1, t_3]$ &
$\omega_{1,2}=[0,1,t_1,t_3, \infty, t_2]$ \\
\hline
$C_1(6)$ & $\omega_{2,1}= [0,1,t_1, \infty, t_2\sha t_3]$ &
$\omega_{2,2}=[0,1,t_1\sha t_2,\infty, t_3]$ \\
  \hline
\end{tabular}
\end{center}
\vspace{0.1in} We therefore have four generators in weight 3. There
are no product relations on $\Mod_{0,6}$, so in order to compute the space
of cell-zeta values, we need only compute the action of the dihedral
group on the four differential forms.  In particular, the order 6 cyclic
generator $0\mapsto t_1\mapsto t_2\mapsto t_3\mapsto 1\mapsto \infty\mapsto 0$
sends
$$\omega_{1,1}\mapsto -\omega_{2,1}-\omega_{2,2},\ \
\omega_{1,2}\mapsto \omega_{1,1},\ \ \omega_{2,1}\mapsto
-\omega_{1,2}-\omega_{2,1},\ \ \omega_{2,2}\mapsto \omega_{2,1}.$$
Thus, letting $X$ denote the standard cell $X_{6,\delta}=
(0,t_1,t_2,t_3,1,\infty)$, we have $\int_X \omega_{1,1}=\int_X
\omega_{1,2}$, $\int_X \omega_{2,1} =\int_X \omega_{2,2}$ and
$2\int_X \omega_{2,2}=\int_X \omega_{1,2}$, so in fact the periods
form a single orbit under the action of the cyclic group of order 6
on $H^\ell(\Modf_{0,S})$. We deduce that the space of periods of
weight 3 is of dimension $1$, generated for instance by $\int
\omega_{2,1}$. Since $\omega_{2,1}$ is the standard form for
$\zeta(3)$, we have
$$ \begin{array}{ccccc}
\zeta(0,1,t_2,\infty,t_1, t_3)& = & \displaystyle{\int_{X} {dt_1
dt_2 dt_3 \over (1-t_2)(t_1-t_3)t_3}} & = &2\, \zeta(3)\ , \nonumber
\vspace{0.04in} \\ \vspace{0.04in} \zeta(0,1,t_1,t_3,\infty, t_2) &=
&\displaystyle{\int_{X} {dt_1 dt_2 dt_3 \over
(1-t_1)(t_1-t_3)t_2}} &=& 2\, \zeta(3)\ , \nonumber \\
\zeta(0,1,t_1,\infty, t_2\sha t_3) &=& \displaystyle{\int_{X} {dt_1
dt_2 dt_3 \over
(1-t_1)t_2t_3}} &=  &\zeta(3)\ , \nonumber \vspace{0.04in}\\
\zeta(0,1,t_1\sha t_2,\infty,t_3) &=& \displaystyle{\int_{X} {dt_1
dt_2 dt_3 \over (1-t_1)(1-t_2)t_3}}& = & \zeta(3)\ ,\nonumber
\end{array}
$$
Note that $\omega_{2,2}$ is the  standard form usually associated to
$\zeta(2,1)$, so that we have recovered the well-known identity
$\zeta(2,1)=\zeta(3)$, which is normally obtained using stuffle,
shuffle and Hoffmann relations on multizetas.

\vspace{.3cm}
\subsubsection{The case $\Mod_{0,7}$} The insertion basis is listed in the
following table. It consists of 22 forms, eleven of which lie in
$C_0(7)$, six of which come from making one insertion into a
convergent $01$ cell-form from $C_0(6)$ (using
$\cal{L}_{1,2}=\{1\sha 2\}$ and $\cal{L}_{2,3}=\{2\sha 3\}$), and
five of which come from making two insertions into the unique
convergent $01$ cell-form from $C_0(5)$ (which also uses
$\cal{L}_{1,2,3}=\{1\sha 2\sha3, 2\sha13 \}$ and
$\cal{L}_{2,3,4}=\{2\sha 3\sha4, 3\sha24\}$).
\begin{center}
\begin{tabular}{|c|c|c|c|}
  \hline
$C_0(7)$ & $[0,1,t_2,\infty,t_3, t_1,t_4]$ & $[0,1,t_1,t_3, \infty,
t_2,t_4]$  & $[0,1,t_1,t_4,t_2,\infty,t_3]$ \\
 &$[0,1,t_2,\infty,t_4, t_1,t_3]$ &  $[0,1,t_1,t_3, \infty, t_4,t_2]$&
 $[0,1,t_2,t_4,t_1,\infty,t_3]$ \\
& $[0,1,t_3,\infty,t_1, t_4,t_2]$ &  $[0,1,t_2,t_4, \infty, t_1,t_3]$&
$[0,1,t_3,t_1,t_4,\infty,t_2]$ \\
& &  $[0,1,t_3,t_1, \infty, t_2,t_4]$ & \\
& &  $[0,1,t_3,t_1, \infty, t_4,t_2]$& \\
 \hline
$C_1(7)$ &  $[0,1,t_2, \infty,t_1, t_3\sha t_4]$ &
$[0,1,t_1,t_4,\infty, t_2\sha t_3]$  & $[0,1,t_1\sha t_2,t_4,\infty,t_3]$\\
 &$[0,1,t_3,\infty,t_1\sha t_2,t_4]$ & $[0,1,t_2\sha t_3, \infty,t_1,
 t_4]$ & $[0,1,t_1,t_3\sha t_4, \infty, t_2]$  \\
  \hline
$C_2(7)$ & $[0,1,t_1,\infty, t_3\sha(t_2,t_4)]$ &
$[0,1,t_1\sha t_2,\infty, t_3\sha t_4]$
& $[0,1,t_2\sha(t_1,t_3), \infty, t_4]$ \\
& $[0,1,t_1, \infty, t_2 \sha t_3 \sha t_4]$ & & $[0,1, t_1\sha
t_2\sha t_3,
\infty, t_4]$\\
 \hline
\end{tabular}
\end{center}

\vspace{0.1in} The standard multizeta forms can be decomposed into sums
of insertion forms as follows:
\begin{equation}
\begin{split}
&{{dt_1dt_2dt_3dt_4}\over{(1-t_1)t_2t_3t_4}}\qquad\
\,=[0,1,t_1,\infty,t_2\sha t_3\sha
t_4]\\
&{{dt_1dt_2dt_3dt_4}\over{(1-t_1)(1-t_2)t_3t_4}}=[0,1,t_1\sha t_2,\infty,
t_3\sha t_4]\\
&{{dt_1dt_2dt_3dt_4}\over{(1-t_1)t_2(1-t_3)t_4}}=[0,1,t_1,t_3,\infty,t_2,t_4]+
[0,1,t_1,t_3,\infty,t_4,t_2]+\\
&\qquad\qquad \qquad\qquad \qquad\qquad
[0,1,t_3,t_1,\infty,t_2,t_4]+ [0,1,t_3,t_1,\infty,t_4,t_2]\\
&{{dt_1dt_2dt_3dt_4}\over{(1-t_1)(1-t_2)(1-t_3)t_4}}=[0,1,t_1\sha
t_2\sha t_3, \infty,t_4]
\end{split}
\end{equation}
In general, the standard multizeta form having factors $(1-t_{i_1}),\ldots,
(1-t_{i_r})$ (with $i_1=1$) and $t_{j_1},\ldots,t_{j_s}$ (with $j_s=n$)
in the denominator is equal to the shuffle form:
\begin{equation}\label{Kontstoins}
[0,1,t_{i_1}\sha\cdots\sha t_{i_r},\infty,t_{j_1}\sha \cdots\sha
t_{j_s}],
\end{equation} 
so to decompose it into insertion forms it is simply
necessary to decompose the shuffles $t_{i_1}\sha\cdots\sha t_{i_r}$
and $t_{j_1}\sha \cdots\sha t_{j_s}$ into linear combinations of
Lyndon insertion shuffles.

Computer computation confirms that the space of periods on $\Mod_{0,7}$ is of
dimension $1$ and is generated by $\zeta(2)^2$. Indeed,
up to dihedral equivalence, there are six product maps on
$\Mod_{0,7}$, given by
\begin{equation}
\begin{cases}
(0,t_1,t_2,t_3,t_4,1,\infty)\mapsto (0,t_1,t_2,1,\infty)\times (0,t_3,t_4,1,
\infty)\\
(0,t_1,t_2,1,t_3,t_4,\infty)\mapsto
(0,t_1,t_2,1,\infty)\times(0,1,t_3,t_4,\infty)\\
(0,t_1,t_2,1,t_3,\infty,t_4)\mapsto
(0,t_1,t_2,1,\infty)\times(0,1,t_3,\infty,t_4)\\
(0,t_1,t_2,1,t_3,\infty,t_4)\mapsto
(0,t_1,1,t_3,\infty)\times(0,t_2,1,\infty,t_4)\\
(0,t_1,t_2,t_3,1,t_4,\infty)\mapsto
(0,t_1,t_2,1,\infty)\times(0,t_3,1,t_4,\infty)\\
(0,t_1,t_2,1,t_3,t_4,\infty)\mapsto
(0,t_1,1,t_3,\infty)\times(0,t_2,1,t_4,\infty)
\end{cases}
\end{equation}
Following the algorithm from section \ref{prodmaps}, we have six associated
relations between the integrals of the 22 cell-forms.   Then, explicitly
computing the dihedral action on the forms yields a further set of
linear equations, and it is a simple matter to solve the entire
system of equations to recover the 1-dimensional solution.  It also
provides the value of each integral of an insertion form as a rational
multiple of any given one; for instance all the values can be computed
as rational multiples of $\zeta(2)^2$.  In particular, we easily recover
the usual identities
$$\zeta(4)={{2}\over{5}}\zeta(2)^2,\ \
\zeta(3,1)={{1}\over{10}}\zeta(2)^2,\ \
\zeta(2,2)={{3}\over{10}}\zeta(2)^2,\ \
\zeta(2,1,1)={{2}\over{5}}\zeta(2)^2.$$

\vspace{.3cm}
\subsubsection{The cases $\Mod_{0,8}$ and $\Mod_{0,9}$}

There are 64 convergent $01$ cell-forms in on $\Mod_{0,8}$, and the
dimension of $H^5(\Mod^{\delta}_{0,8})$ is 144.  The remaining 80
forms are obtained by Lyndon insertion shuffles as follows:
\begin{itemize}
\item{44 forms obtained by making the four insertions
$(t_1\sha t_2,t_3,t_4,t_5)$, $(t_1,t_2\sha t_3,t_4,t_5)$,
$(t_1,t_2,t_3\sha t_4,t_5)$, $(t_1,t_2,t_3,t_4\sha t_5)$  into the
eleven $01$ cell-forms of $\Mod_{0,7}$}
\item{12 forms obtained by the six insertion possibilities
$(t_1\sha t_2\sha t_3,t_4,t_5)$, $(t_2\sha t_1t_3,t_4,t_5)$,
$(t_1,t_2\sha t_3\sha t_4,t_5)$, $(t_1,t_3\sha t_2t_4,t_5)$,
$(t_1,t_2,t_3\sha t_4\sha t_5)$, $(t_1,t_2,t_4\sha t_3t_5)$ into the
two $01$ cell-forms of $\Mod_{0,6}$}
\item{6 forms obtained by the three insertion possibilities
$(t_1\sha t_2,t_3\sha t_4,t_5)$, $(t_1\sha t_2,t_3,t_4\sha t_5)$,
$(t_1,t_2\sha t_3,t_4\sha t_5)$
into the two $01$ cell-forms of $\Mod_{0,6}$}
\item{4 forms obtained by the four insertions
$(t_1\sha t_2\sha t_3,t_4\sha t_5)$,
$(t_2\sha t_1t_3,t_4\sha t_5)$,
$(t_1\sha t_2, t_3\sha t_4\sha t_5)$,
$(t_1\sha t_2, t_4\sha t_3 t_5)$
into the single $01$ cell-form of $\Mod_{0,5}$}
\item{14 forms obtained by the fourteen insertions
$(t_1t_3\sha t_2t_4,t_5)$
$(t_3\sha t_1t_4t_2,t_5)$
$(t_1t_3\sha t_2\sha t_4,t_5)$
$(t_1t_4\sha t_2\sha t_3,t_5)$
$(t_2t_4\sha t_1\sha t_3,t_5)$
$(t_2\sha t_1(t_3\sha t_4),t_5)$
$(t_1\sha t_2\sha t_3\sha t_4,t_5)$
$(t_1,t_2t_4\sha t_3t_5)$
$(t_1,t_4\sha t_2t_5t_3)$
$(t_1,t_2t_4\sha t_3\sha t_5)$
$(t_1,t_2t_5\sha t_3\sha t_4)$
$(t_1,t_3t_5\sha t_2\sha t_4)$
$(t_1,t_3\sha t_2(t_4\sha t_5))$
$(t_1,t_2\sha t_3\sha t_4\sha t_5)$
into the single $01$ cell-form of $\Mod_{0,5}$.}
\end{itemize}

The case of $\Mod_{0,9}$ is too large to give explicitly.  There are
461 convergent $01$ cell-forms, and
dim$\,H^6(\Mod^{\delta}_{0,9})=1089$. An interesting phenomenon
occurs first in the case $\Mod_{0,9}$; namely, this is the first
value of $n$ for which convergent (but not $01$) cell-forms do not
generate the cohomology.  The 1463 convergent cell-forms for
$\Mod_{0,9}$ generate a subspace of dimension 1088.

For $5\le n\le 9$, computer computations have confirmed the
main conjecture, namely:  {\it for $n\le 9$, the weight $n-3$ part
${\cal FC}_{n-3}$ of the formal cell-zeta algebra ${\cal FC}$ is
of dimension $d_{n-3}$, where $d_n$ is given by the Zagier formula
$d_n=d_{n-2}+d_{n-3}$ with $d_0=1$, $d_1=0$, $d_2=1$.}

\end{document}